\documentclass[reqno]{amsart}
\usepackage{epsfig}
\usepackage{color}
\usepackage{amssymb,amsmath,amsthm,amstext,amsfonts}

\usepackage{amssymb,amsmath,amsthm,amstext,amsfonts}
\usepackage{amsmath,amstext,amsthm,amsfonts}
\usepackage{color}
\usepackage[mathscr]{eucal}
\usepackage{dsfont}

\usepackage{psfrag}
\usepackage{url}
\usepackage{epstopdf}
\usepackage{epic,eepic}
\usepackage{upgreek}
\usepackage{amssymb}
\usepackage{mathrsfs}
\usepackage{verbatim}
\usepackage{mathrsfs}
\usepackage{amstext}
\usepackage{amsthm}
\usepackage{amssymb}
\usepackage{graphicx}

\makeatletter \@addtoreset{equation}{section} \makeatother

\renewcommand\thetable{\thesection.\@arabic\c@table}

\theoremstyle{plain}

\newtheorem{theorem}{Theorem}[section]
\newtheorem{proposition}{Proposition}[section]
\newtheorem{lemma}{Lemma}[section]
\newtheorem{corollary}{Corollary}[section]
\newtheorem{definition}{Definition}[section]
\newtheorem{remark}{Remark}[section]
\newtheorem{example}{Example}[section]

\newcommand{\al} {\alpha}

\newcommand{\vep}{\varepsilon}

\newcommand{\R}{\mathbb{R}}

\newcommand{\cC}{\mathcal{C}}

\newcommand{\cF}{\mathcal{F}}

\newcommand{\cU}{\mathcal{U}}

\newcounter{main}

\title{The centralizer of Komuro-expansive flows and expansive $\mathbb R^d$ actions}

\author{W. Bonomo and J. Rocha and  P. Varandas}

\address{Wescley Bonomo,
Universidade Federal do Esp\'irito Santo, CEUNES,
Rodovia Governador Mario Covas, Km 60, 29.932-900, S\~ao Mateus, Brazil.}
              \email{wescley.bonomo@ufes.br}

\address{Jorge Rocha,
Departamento de Matem\'atica, Universidade do Porto,
Rua do Campo Alegre, 687,  4169-007 Porto, Portugal.}
              \email{jrocha@fc.up.pt}

\address{Paulo Varandas,
Departamento de Matem\'atica, Universidade Federal da Bahia,
Av. Ademar de Barros s/n, 40170-110 Salvador, Brazil}
\email{paulo.varandas@ufba.br}

\begin{document}

\maketitle

\begin{abstract}
In this paper we study the centralizer of flows and $\mathbb R^d$-actions on compact Riemannian 
manifolds. We prove that the centralizer of every $C^\infty$ Komuro-expansive flow with non-resonant singularities 
is trivial, meaning it is the smallest possible, and  deduce there exists an open and dense subset of geometric Lorenz 
attractors with trivial centralizer.  We show that $\mathbb R^d$-actions obtained as suspension of $\mathbb Z^d$-actions are expansive if and only if  the same holds for the $\mathbb Z^d$-actions. We also show that
homogeneous expansive $\mathbb R^d$-actions have quasi-trivial centralizers, meaning that it consists
of orbit invariant, continuous linear reparameterizations of the $\mathbb R^d$-action. In particular, 
homogeneous Anosov $\mathbb R^d$-actions have quasi-trivial centralizer.


\end{abstract}

\section{Introduction}

One of the leading problems considered by the dynamical systems community has been to describe the features of most dynamical systems. Based on the pioneering works of Peixoto and Smale, the program proposed by Palis
in the nineties has constituted a route guide for a global itemization of the space of dynamical systems.
This program, that proposed the complement of uniform hyperbolicity as the space of diffeomorphisms that are approximated by those exhibiting either heteroclinic tangencies or heteroclinic cycles, was carried out much successfully
in the $C^1$-topology, where perturbation tools like the closing lemma, Franks' lemma, connecting lemma or ergodic  closing lemma are available (see \cite{Ma1} and \cite{Hay}).

In the seminal papers \cite{Sm90,Sm}, Smale conjectured that most dynamical systems 
should have trivial centralizer, a hard problem not yet completely understood. Given a $\mathcal{C}^r$-diffeomorphism
$f$ on a compact manifold $M$, its $\mathcal C^r$-centralizer $\mathcal{Z}^r(f) = \{g \in \text{Diff}^{\,r}(M): f \circ g = g \circ f\}$ is a subgroup of $\text{Diff}^{\,r}(M)$, $r \geq 1$. In some sense, the centralizer reflects symmetries of the dynamics which typically should be rare. The problem of the centralizer is related e.g. with the embedding of maps as time-1 maps of flows
\cite{Palis} or the problem of differentiability of conjugacies ~\cite{Yoccoz}. 
In the discrete-time setting some results in the direction of a positive answer to Smale's conjecture are known. 
In the seventies, Walters~\cite{Wa70} proved that
expansive homeomorphisms have discrete centralizers, and Kopell~\cite{Ko70} established that an open and dense subsets of $C^r$ ($r\ge 2$) 
circle diffeomorphisms have trivial centralizer. Palis and Yoccoz~\cite{PY89} proved that an open and dense subset of $C^\infty$-
Axiom A diffeomorphisms with the strong transversality property and that admit a periodic sink have trivial centralizers. Similar results were established by
Rocha~\cite{Ro93} and Fisher~\cite{Fi09} for surface Axiom A diffeomorphisms and diffeomorphisms with codimension one hyperbolic attractors, respectively. 
In the case of dynamical systems beyond the scope of uniform hyperbolicity, there exists a residual subset of certain classes 
of $C^r$, $r\ge 1$, partially hyperbolic diffeomorphisms with discrete centralizer (cf. Burslem~\cite{Bu04}).
Moreover, it is worth mentioning that there are open sets of surface Anosov diffeomorphisms whose
$C^0$-centralizer is discrete but not trivial ~\cite{Ro05},
and that the subset of diffeomorphisms with trivial centralizer in the $C^1$ topology has nonempty interior ~\cite{BF}.
Finally, there exists a $C^1$-residual subset of $\text{Diff}^1(M)$ with trivial centralizer \cite{BCW},
and the set of $C^1$-diffeomorphisms with trivial centralizer is not open and dense \cite{BCVW}.

In the time-continuous setting the picture is still much more incomplete. In opposition to the discrete time setting the centralizer is never discrete. Indeed, using that the flow commutes with itself, the centralizer of a flow clearly contains 
a continuum. 
Some advances to establish the counterpart of Smale's conjecture for Anosov, $C$-expansive and Axiom A flows 
with the strong transversality were obtained in ~\cite{Kato},~\cite{oka}, ~\cite{Sad2} and~\cite{JH91}.
Such a description of the centralizer can be given in terms of the flow or of the generating vector field.
However, flows with some weak hyperbolicity and where singularities accumulated by regular orbits 
of the flow have not yet been considered, and this is a goal of the present work.
These include important classes of flows as three-dimensional $C^1$-robustly transitive flows with singularities,
often referred as singular-hyperbolic flows. These are partially hyperbolic flows with an invariant splitting in a 
one-dimensional contracting and a two-dimensional volume expanding invariant subbundles for the vector field $X$ or
the vector field $-X$. Any singular-hyperbolic flow without singularities is an Anosov flow. Hence, if 
the manifold does not support Anosov flows then the singular set of a singular-hyperbolic flow is non-empty.
Moreover, singular-hyperbolic attractors include the important classes of examples of geometric Lorenz attractors,
introduced by Afraimovich, Bykov and Shilnikov \cite{lorAfraimovich} and Guckenheimer, Williams \cite{lorGuckenheimer} 
to model the chaotic attractor proposed by E. Lorenz~\cite{Lorenz}.
We refer the reader to ~\cite{vitorzeze} for precise definitions and a large account on singular-hyperbolicity.

Our first purpose in the current article is to describe the centralizer of a class of flows that contains the important classes of geometric Lorenz attractors. These admit a weak form of expansiveness (the so called Komuro-expansiveness) which is compatible with the coexistence of regular and singular orbits in the same transitive piece of the non-wandering set. 
We prove first that for a Komuro-expansive flow any commuting flow is a continuous linear reparameterization of the original flow. 
If, in addition, the singularities satisfy an (open and dense) non-ressonance condition then the centralizer of such a vector 
field $X$ consists of the vector fields of the form $cX$ for some $c\in \mathbb R$ on the closure of the stable manifolds of the singularities (cf. Theorem~\ref{thm:flows} and Corollary~\ref{cor:centr}). 
As a byproduct of these results we conclude that the orbits of $\mathbb R^d$-actions that admit some expansive 
element are indeed one-dimensional, a fact that holds e.g. for Anosov actions 
(cf. Corollary~\ref{cor:actions}). 
As a second purpose we also obtain a systematic treatment of the centralizer of expansive $\mathbb R^d$-actions.
In the case of $\mathbb R^d$-actions that are suspension of $\mathbb Z^d$-actions, expansiveness is either a common feature or it fails for both actions (cf. Theorem~\ref{thm:suspensions}). 
 We also prove that the centralizer of expansive ``typical homogeneous'' $\mathbb R^d$-actions is also reduced to continuous reparameterizations. We refer the reader to Theorem~\ref{centacaoespansiva} for the precise statement.

This paper is organized as follows. In Section~\ref{sec:statements} we introduce some
definitions and state the main results of this paper.
Section~\ref{sec:examples} is devoted to present some examples and a wider discussion and comparison of
our results with other notions of expansiveness for flows.
In Section~\ref{sec:flows} we study the centralizer of expansive flows with singularities.
The results on the expansiveness properties and centralizer of $\mathbb R^d$-actions are given along
Sections~\ref{sec:actions1} and ~\ref{sec:actions2}.

\section{Preliminaries and statement of the main results}\label{sec:statements}

\subsection{Preliminaries}

In this subsection we shall introduce some definitions and recall some necessary background, with the intention of making the text as self-contained as possible. 
Throughout we let $M$ be a compact Riemannian manifold.

\subsubsection{Uniform hyperbolicity}\label{unifh}

Given a $C^r$-diffeomorphism $f$ on $M$, $r\ge 1$, a set $\Lambda \subset M$ is \emph{uniformly hyperbolic} 
if there is a $Df$-invariant splitting $T_\Lambda M = E^s \oplus E^u$ and constants $C>0$ and $\lambda\in (0,1)$ so that 
$$
\| Df^n(x)\mid_{E^s_x} \| \le C \lambda^n 
	\quad\text{and}\quad
	\| (Df^n(x)\mid_{E^u_x})^{-1} \| \le C \lambda^n 
$$
for every $x\in \Lambda$ and $n\ge 1$. We refer to $T_{\Lambda} M = E^s \oplus E^u$ as the \emph{hyperbolic splitting} associated to $f$ and $\Lambda$.

Let ${\text{Per}(f)}$ denote the set of periodic points and $\Omega(f)$ denote the non-wandering set of $f$.
A $C^1$-diffeomorphism $f$ is called \emph{Axiom A} if $\overline{\text{Per}(f)} = \Omega(f)$ and $\Omega(f)$ is a uniformly hyperbolic set. The diffeomorphism $f$ is \emph{Anosov} if the manifold $M=\Lambda$ is a hyperbolic set for $f$.

The natural counterpart for flows is defined as follows. Given a vector field $X\in \mathfrak{X}^r(M)$, $r\ge 1$, let  $(\varphi_t)_{t\in \mathbb R}$ denote the $C^r$-flow on $M$ generated by $X$. 
Recall that $\sigma \in M$ is a \emph{hyperbolic singularity} for $X\in \mathfrak{X^1}(M)$ provided
$X(\sigma)=0$ and $DX(\sigma)$ does not contain any purely imaginary eigenvalue. 
Furthermore, we say that a hyperbolic singularity $\sigma$ for $X\in \mathfrak{X^1}(M)$ is \emph{non-resonant} if 
the eigenvalues $\al_1, \dots, \al_k \in \mathbb C$ of $DX(\sigma)_{|E^{u}_\sigma}$ 
(resp. the eigenvalues $\beta_1, \dots, \beta_m \in \mathbb C$ of $DX(\sigma)_{|E^{s}_\sigma}$)
are all distinct and do not satisfy any relation of the form $\text{Re}(\alpha_i)=\sum_{j \neq i} n_j \text{Re}(\alpha_j)$
(resp. $\text{Re}(\beta_i)=\sum_{j \neq i} n_j \text{Re}(\beta_j)$)
for some non-negative integers $n_j$ so that $\sum_{j=1}^k n_j \ge 2$. 
Observe that since we consider the eigenvalues of stable and unstable bundles independently, the later 
corresponds to the singularity being separately non-resonant. Furthermore, in this case the ressonance conditions consist of
\emph{finitely} many algebraic closed equations and, consequently, are satisfied by an open and dense
subset of linear vector fields. 

Given a compact $(\varphi_t)_{t\in \mathbb R}$-invariant non-singular set $\Lambda\subset M$, we say that $\Lambda$ is a \emph{hyperbolic set} 
for $(\varphi_t)_{t\in \mathbb R}$ if there exists a $D\varphi_t$-invariant splitting $T_\Lambda M = E^s \oplus E^0 \oplus E^u$ so that:
(a) $E^0$ is one dimensional and generated by the vector field, (b) there are constants 
$C>0$ and $\lambda\in (0,1)$ so that 
$$
\| D \varphi_t (x)\mid_{E^s_x} \| \le C \lambda^t 
	\quad\text{and}\quad
	\| (D\varphi_t (x)\mid_{E^u_x})^{-1} \| \le C \lambda^t 
$$
for every $x\in \Lambda$ and $t\ge 0$. 

Given a $C^1$-flow $\varphi=(\varphi_t)_t$ we denote by $Sing(\varphi)$ the singularities of $\varphi$ and by $Crit(\varphi)$
the set of all critical elements, formed by singularities and closed orbits for the flow $\varphi$.
A flow $(\varphi_t)_t$ is called \emph{Axiom A} if 
$\overline{\text{Crit}(\varphi)}=\Omega(X)$ and the non-wandering set $\Omega(X)$ is a uniformly hyperbolic set.
The flow $(\varphi_t)_t$ is \emph{Anosov} if $\Lambda=M$ is a hyperbolic set.

We say that $\Phi : \mathbb R^d \times M \rightarrow M$ is a $\mathcal{C}^r$-action on a compact Riemannian manifold $M$ if $\Phi_v:=\Phi(v, \cdot) \colon M \to M$ is a $C^r$ diffeomorphism and $\Phi_{v+u}= \Phi_{v}\circ \Phi_{u}$ for every $v,u\in \mathbb R^d$. Following \cite{BM}, we say that a $\mathcal{C}^r$-action $\Phi : \mathbb R^d \times M \rightarrow M$ on a compact Riemannian manifold $M$ is an \emph{Anosov action} if there exists 
$v \in \mathbb R^d$ such that the diffeomorphism $\Phi_{v}$ admits a continuous $D\Phi_v$-invariant decomposition $TM=E_v^s \oplus T\Phi \oplus E_v^u$ where $T\Phi$ denotes the tangent space to the orbits of $\Phi$ and 
there are constants $C>0$ and $\lambda\in (0,1)$ so that 
$$
\| D \Phi_v^n (x)\mid_{E^s_x} \| \le C \lambda^n 
	\quad\text{and}\quad
	\| (D\Phi_v^n (x)\mid_{E^u_x})^{-1} \| \le C \lambda^n 
$$
for every $x\in M$ and $n\ge 0$. The diffeomorphism $\Phi_v$ is called an \emph{Anosov element}.

Let $\cF$ denote the orbit foliation of $\Phi$ and let $\cF(x)$ denote the leaf of the foliation containing the point $x$. 
It follows from \cite[Theorem~7.2]{HPS77} that $(\Phi_v,\cF)$ is a 
plaque expansive diffeomorphism: there exists $\overline{\delta} > 0$ such that if $(x_n)_{n \in \mathbb{Z}}$ and $(y_n)_{n \in \mathbb{Z}}$ are $\overline\delta$-pseudo-orbits  preserving $\cF$ (ie. $d(\Phi_v(x_n),x_{n+1})<\overline{\delta}$,
 $d(\Phi_v(y_n),y_{n+1})<\overline{\delta}$, $\Phi_v(x_n) \in \cF_\delta(x_{n+1})$ and $\Phi_v(y_n) \in \cF_\delta(y_{n+1})$ for all $n\in \mathbb Z$) and the pair of points $x_n,y_n$ remains $\overline
 \delta$-close for all $n$, then $y_n \in \mathcal{F}_\delta(x_n)$ for every $n\in \mathbb Z$.

\subsubsection{Expansiveness}\label{sec:expansiveness}

First we shall recall the notion of expansiveness in the discrete time setting. Given a homeomorphism $f\in \text{Homeo}(M)$ and a compact invariant set $\Lambda \subset M$, we say that $f$ is \emph{expansive} in $\Lambda$ if there exists $\delta > 0$ so that for all $x, \, y \in \Lambda$ satisfying $d(f^n(x),f^n(y)) \le \delta$ for every 
$n\in \mathbb Z$ one has $x=y$.
In the time-continuous setting of flows, due to the possible presence of singularities, there are several notions of 
expansiveness (see e.g.~\cite{BW72,oka90}). We recall some of these notions, starting by the one introduced by Bowen and Walters~\cite{BW72}.

\begin{definition} Let $(M, \, d)$ be a compact metric space, $\varphi: \mathbb{R} \times M \rightarrow M$ be a continuous flow, and $\Lambda \subseteq M$ be a compact $\varphi$-invariant set. We say that the flow $\varphi$ is \emph{$C$-expansive in $\Lambda$} if for any $\varepsilon > 0$ there exists $\delta > 0$ so that if $x, \, y \in \Lambda$ and $d(\varphi_t(x), \, \varphi_{h(t)}(y)) < \delta$ for all $t \in \mathbb{R}$ for some continuous function $h: \mathbb{R} \rightarrow \mathbb{R}$ satisfying $h(0) = 0$, then $y = \varphi_{t_0}(x)$ for some $|t_0| < \varepsilon$. 
\end{definition}

The later means that, for expansive flows, orbits of two points $x, \, y$ by the flow that always remain close to each other
(up to reparameterization) do coincide. Singularities of a $C$-expansive flow are necessarily isolated points \cite[Lemma~1]{BW72}. Moreover, $C$-expansive flows on connected manifolds do not admit singularities.
A weaker notion, as follows, was introduced later by Keynes and Sears~\cite{Keynes}.

\begin{definition}\label{kexpansiveness} 
Let $(M, \, d)$ be a compact metric space, $\varphi: \mathbb{R} \times M \rightarrow M$ be a continuous flow, and $\Lambda \subseteq M$ be a compact $\varphi$-invariant set. We say that the flow $\varphi$ is  
\emph{$K$-expansive in $\Lambda$}  if for any $\varepsilon > 0$ there exists $\delta > 0$ so that
if $x, \, y \in \Lambda$ and $d(\varphi_t(x), \, \varphi_{h(t)}(y)) < \delta$ for all $t \in \mathbb{R}$ 
for some increasing homeomorphism $h: \mathbb{R} \rightarrow \mathbb{R}$ with $h(0) = 0$ then
$y = \varphi_{t_0}(x)$ for some $|t_0| < \varepsilon$. 
\end{definition}

Although the later is more general, these two notions are indeed equivalent in the case that $M$ is a compact Riemannian manifold (see e.g. \cite{vitorzeze}). Motivated by the analysis of flows with non-isolated singularities in the non-wandering set as the classical geometric Lorenz attractors, Komuro~\cite{komuro} introduced a more general notion of expansiveness  that we now describe.

\begin{definition}\label{expansiveflow}
Let $(M, \, d)$ be a compact metric space, $\varphi: \mathbb{R} \times M \rightarrow M$ be a continuous flow, and $\Lambda \subseteq M$ be a compact $\varphi$-invariant set. We say that the flow $\varphi$ is \emph{Komuro-expansive in $\Lambda$} if for any $\varepsilon > 0$ there exists $\delta > 0$ so that if $x, \, y \in \Lambda$ and $d(\varphi_t(x), \, \varphi_{h(t)}(y)) < \delta$ for every $t \in \mathbb{R}$ and for some increasing homeomorphism $h: \mathbb{R} \rightarrow \mathbb{R}$ then there is $t_0\in \mathbb R$ such that $\varphi_{h(t_0)}(y) \in \varphi_{[t_0 - \varepsilon, \, t_0 + \varepsilon]}(x)$. Here, as usual,
$\varphi_{[t_0 - \varepsilon, \, t_0 + \varepsilon]}(x):=\{ \varphi_{t}(x) \colon t\in [t_0 - \varepsilon, \, t_0 + \varepsilon]\}$.
\end{definition}

Sometimes Komuro-expansive flows are simply called expansive in literature, and we keep this nomenclature.
It follows from the previous definitions that
\begin{equation}\label{compexpansiveness}
C\textrm{-expansiveness} \Rightarrow \; K\textrm{-expansiveness} \Rightarrow \textrm{ expansiveness}.
\end{equation}
On the converse direction, these three notions of expansiveness are equivalent for flows without singularities
(cf. \cite[Theorem~A]{oka90}) but may differ for flows with singularities. 
If not explicited stated otherwise we assume $\Lambda = M$, meaning expansiveness holds in the whole manifold. 
It is well known that  uniformly hyperbolic flows are $C$-expansive and, for that reason, expansive diffeomorphisms and flows contain as 
special classes of examples the cases of Anosov diffeomorphisms and Anosov flows, respectively.

In the case of $\mathbb{R}^d$-actions, each orbit of a point $x$ in the manifold $M$ is an immersed submanifold of
dimension at most $d$.  In the case that the foliation formed by the orbits of points consist of submanifolds with 
different dimensions, this encloses the same kind of topological difficulties (in a wider range of possibilities) 
caused by the presence of singularities for flows.

\begin{definition}\label{acaoexpansivacontinua} 
Let $M$ be a compact metric space and $\Phi: \mathbb{R}^d \times M  \rightarrow M$ be a continuous action. We say that $\Phi$ is \emph{expansive} if for any $\varepsilon > 0$, there exists $\delta > 0$ so that if $x, \, y \in M$ satisfy $d(\Phi_v(x), \, \Phi_{h(v)}(y)) < \delta$ for every $v \in \mathbb{R}^d$ with respect to a continuous function $h: \mathbb{R}^d \rightarrow \mathbb{R}^d$ so that $h(0) = 0$,  then
$y = \Phi_{v_0}(x)$ for some $\|v_0\| < \varepsilon$. In particular, $y$ belongs to the orbit  of $x$ by $\Phi$.
\end{definition}

The study of the geometry and topology of the foliation by orbits of $\mathbb R^d$-actions is a hard problem 
and encloses much more difficulties than the case of flows. For instance, 
in opposition to  the case of vector fields, we do not expect 
all expansive $\mathbb R^d$-actions on compact connected Riemannian manifolds to  be homogeneous.

\subsubsection{Centralizers}

Given $r\ge 0$ and a diffeomorphism $f\in \text{Diff}^{\,r}(M)$ the \emph{centralizer of $f$}
is the set 
$$
\mathcal{Z}^r(f)=\{ g\in \text{Diff}^{\,r}(M) \colon f \circ g = g \circ f \}
$$
where, by some abuse of notation, we let $\text{Diff}^{\,0}(M)$ denote the space of homeomorphisms.
The definition for time-continuous dynamical systems is analogous.
Given $r \geq 0$, let $\mathcal F^r(M)$ denote the space of $C^r$-flows on $M$.
Given a flow $\varphi = (\varphi_t)_{t\in\mathbb R} \in \mathcal F^r(M)$, the \emph{centralizer of $\Phi$} is defined as
$$
\mathcal Z^r (\varphi)
	=\{ \psi=(\psi_t)_{t\in \mathbb R} \in \mathcal F^r(M) 
		\colon \psi_s \circ \varphi_t = \varphi_t \circ  \psi_s , \forall \, s, t\in \mathbb R\}.
$$
It is clear from the previous definition that flows obtained as reparameterizations of a flow $\varphi$ 
 belong to $\mathcal Z^r (\varphi)$. For that reason, the centralizer of a flow is never a discrete subgroup.
In the case of smooth flows, the previous characterization of centralizer has a dual formulation in terms of vector 
fields. Given $r\ge 1$ and $X\in \mathfrak{X}^r(M)$, one can define the \emph{centralizer of the vector field $X$} by
$$
\mathcal Z^r (X)
	=\{ Y \in \mathfrak{X}^r(M) 
		\colon [X,Y]=L_Y X= 0\},
$$
where $L_Y X$ denotes the Lie derivative of the vector field $X$ along $Y$.

\begin{definition} 
Given $r\ge 0$, we say a flow $\varphi=(\varphi_t)_t \in \cF^r(M)$ has \emph{quasi-trivial  centralizer} if for any $\psi \in \mathcal{Z}^r(\varphi)$  there exists a $C^r$-function $A: M \rightarrow \mathbb{R}$ so that 
\begin{itemize}
\item[(i)] (orbit invariance) $A(x) = A(\varphi_t(x))$ for every $(t, \, x) \in \mathbb{R} \times M$, and 
\item[(ii)] $\psi_t(x) = \varphi_{A(x)t} \, (x)$ for every $(t, \, x) \in \mathbb{R} \times M$.
\end{itemize}
In the case that the reparameterizations $A$ are necessarily constant then we say the centralizer is \emph{trivial}.
Dually, we say that $X \in \mathfrak{X}^r(M)$ has \emph{quasi-trivial centralizer} if for any $Y \in \mathcal{Z}^r(X)$ there
exists $A \in \mathcal{C}^r(M, \, \mathbb{R})$ so that $Y = A \cdot X$ and $X(A) = 0$. If $A$ is constant then we say that the centralizer is trivial.
\end{definition}

Observe that the previous notions for vector fields and flows are dual. On the one hand, if $X(h) = 0$ for some $h: M \to \mathbb R$ then $h$ is constant along the orbits of the flow. On the other hand, if $Y = h \cdot X$ 
and $h$ is constant along the orbits of the flow $(X_t)_t$ generated by $X$ then the flow $(Y_t)_t$ generated by $Y$
satisfies $Y_t(x) =  X_{h(x) t} (x)$ for every $t \in \mathbb{R}$ 
and $x \in M$.
The centralizer of a $\mathbb R^d$-action is defined similarly.

\begin{definition}
Given a $C^r$-action $\Phi: \mathbb{R}^d \times M \rightarrow M$, 
we define its \emph{centralizer} as the set 
$\mathcal{Z}^r(\Phi) = \{\Psi : \mathbb{R}^d \times M \rightarrow M : \; \Phi_v \circ \Psi_u = \Psi_u \circ \Phi_v $ for all $v, \, u \in \mathbb{R}^d\}$. We say that $\Phi$ has a \emph{quasi-trivial centralizer} if for any $\Psi \in \mathcal{Z}^r(\Phi)$ there exists a $C^r$-map $A: M \rightarrow \mathcal M_{d\times d}(\mathbb{R})$ satisfying $A(x) = A(\Phi_v(x))$ for every $v\in\mathbb R^d$
and $x\in M$ and so that $\Psi_v(x) = \Phi(A(x) v, \, x)$  for every $(v, \, x) \in \mathbb{R}^d \times M$. 
\end{definition}

\subsection{Statement of the main results}\label{sec:state}
This subsection is devoted to the statement of the main results.

\subsubsection{Centralizers of Komuro-expansive flows}
It is known that $C$-expansive flows on compact and connected metric spaces have quasi-trivial centralizer 
~\cite{oka} (the author used the nomenclature of ``unstable centralizer''). Our first result is an extension of the aforementioned results for the broader class of expansive flows.

\begin{theorem}\label{thm:flows}
Let $\varphi$ be a $C^\infty$ flow on a compact, connected Riemannian manifold $M$ and let 
$\Lambda \subset M$ be a compact $\varphi$-invariant set such that $\varphi$ is Komuro-expansive in $\Lambda$. 
If all the singularities of $\varphi_{|\Lambda}$ are hyperbolic and non-resonant then for any $\psi \in \mathcal{Z}^{\infty}(\varphi_{|\Lambda})$ there exists a continuous map $A: \Lambda \rightarrow \mathbb{R}$, constant along the orbits of $\varphi_{|\Lambda}$ (meaning $A(x) = A(\varphi_t(x))$ for every $x \in \Lambda$ and $t \in \mathbb{R}$) so that $\psi_t(x) = \varphi_{A(x) t}(x)$  for every $(t, \, x) \in \mathbb{R} \times \Lambda$.
\end{theorem}

Some comments are in order.
Firstly, the $C^\infty$ regularity assumption in the previous theorem is not used in full strength.
Indeed, the argument in the proof of the theorem can be divided two main steps: (i) the orbit of a regular point 
by an element in the centralizer is a reparameterization of the original trajectory, and (ii) the reparameterization obtained at regular orbits extend continuously to singularities. 
Our strategy combines the linearization at hyperbolic singularities (which requires sufficiently regularity of the vector field 
given in terms of conditions on the eigenvalues of the singularities as in Sternberg's linearization results) together 
with the characterization due to Kopell~\cite{Ko70} 
that the $C^r$-centralizer of their linear part is formed only by linear transformations provided that $r$ is 
large enough (we refer the reader to Subsection~\ref{sec:extension} for the details).  The $C^\infty$ assumption 
allows us to simplify the statement. Moreover, the centralizer can be proved trivial in the case that stable/unstable manifolds
of singularities are dense in the phase space. 
Secondly, we observe that a version of Theorem~\ref{thm:flows} for volume preserving flows also holds. This follows straightforwardly 
from the arguments used in the proof of Theorem~\ref{thm:flows} using linearization results for volume preserving vector fields
(see e.g. ~\cite{llave} and references therein). 
Thirdly, Komuro-expansive flows do not form a $C^1$-open set of flows. Nevertheless, the $C^1-$interior of 
the set of Komuro-expansive flows is not empty and contains the important classes of Anosov  
and singular-hyperbolic flows. Since our results apply to proper invariant sets, in particular we deduce there exists 
an open and dense subset of $C^\infty$- geometric Lorenz attractors 
with trivial centralizer (cf. Example~\ref{ex:Lorenz}).
Finally, recalling the duality between commuting flows and vector fields, Theorem~\ref{thm:flows} admits the following reformulation: if $X \in \mathfrak{X}^{\infty}(M)$ generates a singularity-free Komuro-expansive flow on a compact and connected Riemannian manifold $M$ 
then for any $Y \in \mathcal{Z}^{\infty}(X)$ there exists $h\in C^\infty(M,\mathbb R)$ so that $Y=h \cdot X$ and $X(h)=0$.

\smallskip

Our strategy relies on the quasi-triviality of the centralizer on the topological basin of attraction of attractors.
Using spectral decomposition in finitely many basic pieces,  Sad~\cite[Theorem~B]{Sad2} proved that there is an open 
and dense subset $A_{\tau}^{'}$ of the $C^\infty$-Axiom A vector fields with the strong transversality condition so that $\mathcal{Z}^{\infty}(X) = \{c X: \, c \in \mathbb{R}\}$ for every $X \in A_{\tau}^{'}$. The following can be understood
as an extension of \cite{Sad2}, where expansiveness and the (open and dense) non-ressonance condition replaces
the uniformly hyperbolic assumption of \cite{Sad2}.

\begin{corollary}\label{cor:centr}
Let $\varphi$ be an expansive $C^\infty$-flow on a compact and connected Riemannian manifold $M$. Assume that all the singularities $\text{Sing}(\varphi)$ are hyperbolic and non-resonant. If
$$
\Lambda= \overline{\bigcup_{\sigma \in \text{Sing}(\varphi)} W^s(\sigma) }  \;\;\;\;\;\; \Big(\textrm{or } \;\; \Lambda= \overline{\bigcup_{\sigma \in \text{Sing}(\varphi)} W^u(\sigma) } \Big)
$$ 
then $\mathcal{Z}^{\infty}(\varphi\mid_\Lambda)$ is trivial. In other words, if $\psi \in \mathcal{Z}^{\infty}(\varphi\mid_\Lambda)$ then there exists $c\in \mathbb R$ so that $\psi_t(x)=\varphi_{ct}(x)$ for every $t\in \mathbb R$ and $x\in \Lambda$.
\end{corollary}

\smallskip

\subsubsection{Centralizers of $\mathbb R^d$-actions}

Our previous results have implications for the study of the centralizer of smooth $\mathbb R^d$-actions that admit expansive elements. For example, the following is a consequence of Theorem \ref{thm:flows}.

\begin{corollary}\label{cor:actions}
Take $d\ge 1$ and let $\Phi: \mathbb R^d \times M \to M$ be a continuous $\mathbb R^d$-action on a compact Riemannian manifold $M$. If there exists $v\in \mathbb R^d$ so that $(\Phi_{tv})_{t\in \mathbb R}$ is an expansive flow then the orbits of $\Phi$ are one-dimensional and coincide with the orbits of a flow. 
\end{corollary}

By the previous corollary, if an $\mathbb R^d$-action has some orbit of dimension larger than one then there exists no $v\in \mathbb R^d$ so that $(\Phi_{tv})_{t\in \mathbb R}$ is an expansive flow. 
In what follows we shall introduce the notion of homogeneous $\mathbb R^d$-actions.

\begin{definition}
Let $M$ be a compact Riemannian manifold. We say that the $\mathbb{R}^d$-action $\Phi: \mathbb R^d \times M \to M$ 
is \emph{homogeneous} (or \emph{locally free}) if all orbits by $\Phi$ are submanifolds on $M$ with dimension equal to $d$. 
\end{definition}

We observe that a flow is homogeneous if and only it has no singularities. In particular, not every manifold admits homogeneous $\mathbb{R}^d$-actions (e.g.  every $C^1$-flow on $\mathbb S^2$ admits singularities, by 
Poincar\'e-Bendixson theorem). On the other hand, the only manifolds that support homogeneous $\mathbb R^d$-actions 
are torus $\mathbb T^n$ ($n\ge d$) (see Lemma~2 at Subsection 49 of \cite{Arnold}), 
and the space of homogeneous $\mathbb{R}^d$-actions forms a open subset of all 
$\mathbb{R}^d$-actions.
Indeed, if $(e_1, \dots, e_d)$ denotes a basis of $\mathbb R^d$, $\Phi: \mathbb{R}^d \times M \rightarrow M$ 
is a smooth $\mathbb R^d$-action, and $\varphi_{e_i}:=(\Phi_{te_i})_{t\in\mathbb R}$ denotes the canonical flow generated by the direction 
$e_i$ 
\[
\begin{array}{rccc}
\varphi_{e_i}: & \mathbb R \times M & \to &  M \\
	& (t,x) & \mapsto & \Phi(t  e_i, \, x)
\end{array}
\]
then it is not hard to check that $\Phi$ is homogeneous if and only if the vector fields 
$X_{e_i}(x) = \frac{d}{dt} \varphi_{e_i} (t,x) \! \mid_{t=0}$ are linearly independent at all points of $M$, which is clearly
an open condition.  In what follows we describe the centralizer of expansive and homogeneous $\mathbb{R}^d$-actions. 
 
 \begin{theorem}\label{centacaoespansiva}
Let $\Phi: \mathbb{R}^d \times \mathbb T^n \rightarrow \mathbb T^n$ be a $C^1$ action, $n\ge d$. 
If $\Phi$ is expansive and homogeneous then $\mathcal{Z}^1(\Phi)$ is quasi-trivial, i.e.,
for every $\Psi \in \mathcal{Z}^1(\Phi)$ there exists a $C^1$-map $A: \mathbb T^n \rightarrow \mathcal M_{d\times d}(\mathbb{R})$ satisfying $A(x) = A(\Phi_v(x))$ for every $v\in\mathbb R^d$ and so that $\Psi_v(x) = \Phi(A(x) v, \, x)$  for every $(v, \, x) \in \mathbb{R}^d \times \mathbb T^n$. 
\end{theorem}

The previous result is the counterpart of \cite{oka} for $\mathbb R^d$-actions.
We also obtain a geometrical interpretation for the reparameterization $A(\cdot)$ obtained above 
(see Proposition~\ref{Acomblinear}) and prove that Anosov actions also have quasi-trivial centralizer 
(see Example~\ref{corAn}).
The description of the centralizer of non-homogeneous, expansive $\mathbb R^d$-actions encodes difficulties 
similar to the ones for flows where singular and non-singular orbits coexist. The strategy used in the case of flows 
with singularities can probably be applied similarly in the case that the set of singular orbits (ie. of dimension smaller 
than $d$) has empty interior in $M$. More generally, it is not hard to see that elements in the centralizer of any $\mathbb R^d$-action preserve orbits of the same dimension but it is unclear if these are reparameterizations of the original action. 
We also establish a characterization of expansive $\mathbb R^d$-actions obtained as suspensions of $\mathbb Z^d$-actions
(Theorem~\ref{thm:suspensions}). Since the precise statement of this result requires many extra definitions we will state and
prove it in Section~\ref{sec:actions1}).

\section{Centralizer of expansive flows}\label{sec:flows}

This section is devoted to the proof of Theorem~\ref{thm:flows}. We subdivide the proof in subsections for making the exposition clearer.
Let $\varphi$ be a $C^\infty$ flow defined in a compact manifold $M$ and $\Lambda \subset M$ 
be a compact $\varphi$-invariant subset on which the flow is expansive and all singularities are hyperbolic and non-resonant. First we prove that all periodic orbits on $\Lambda$, if they exist, have their periods larger than some uniform lower bound (Lemma \ref{e0}). 
This is enough to  obtain tubular flowboxes of uniform size $\mu$, at each regular point, 
and to use expansiveness to guarantee that every flow $\psi \in \mathcal{Z}^{\infty}(\varphi_{|\Lambda})$ is locally
a reparameterization of the flow $\varphi$ on $\Lambda \setminus \text{Sing}(\varphi_{|\Lambda})$ 
and that such a local reparameterization is unique and  defined for all time in $[-\mu, \, \mu]$ (cf. Lemma~\ref{singepsilon1}). 
Then we borrow the strategy of \cite{oka} to prove that although $\Lambda \setminus Sing(\varphi_{|\Lambda})$ may be non-compact the local reparameterizations of the orbits can be uniquely extended  to the real line $\mathbb{R}$ on all regular points $x\in \Lambda \setminus Sing(\varphi_{|\Lambda})$ (cf. Proposition~\ref{extencao}). 
Finally, we show that the reparameterizations are linear and constant along orbits of regular points of $\varphi_{|\Lambda}$. 
Using the linearization of the singularities together with a version of Kopell's description of the centralizer of linear flows  (Lemma~\ref{kopellflows}) we conclude that the reparameterizations can be continuously extended to the singularities  (hence, are globally defined in $\Lambda$) and that these reparameterizations are smooth off of $Sing(\varphi\mid_\Lambda)$.
Throughout this section, and for notational simplicity, we use the notation $\varphi_t$ instead of $(\varphi\mid_\Lambda)_t$.

\subsection{Bound on length of periodic orbits}

In next lemma we prove the existence of positive infimum for the period of periodic orbits of regular points of $\varphi_{|\Lambda}$, in case they exist. We will need the tubular neighborhood theorem for vector fields. 

\begin{lemma}
\label{tubular}
Let $M$ be a compact Riemannian manifold of dimension $n$. 
Given $X\in \mathfrak{X}^1(M)$ and a regular point $x\in M$ there exists $\delta=\delta_x>0$, an open neighborhood 
$U^\delta_x$ of $x$ (called tubular neighborhood), and a $C^1$-diffeomorphism $\Psi_x: U^\delta_x \to (-\delta,\delta)\times B(x,\delta) \subset \mathbb R \times \mathbb R^{d-1}$ such that the vector field $X$ on $U^\delta_x$ is the pull-back of the vector field $Y:= (1,0, \dots, 0)$ on $(-\delta,\delta)\times B(x,\delta)$, that is, 
$
Y=(\Psi_x)_* X:=D(\Psi_x)_{\Psi_x^{-1}}X(\Psi_x^{-1}).
$ 
In particular $Y_t(\cdot)=\Psi_x(X_t(\Psi_x^{-1}(\cdot)))$ for every $|t|< \delta$. 
\end{lemma}

\begin{lemma}\label{e0} Let $\varphi$ be a $\mathcal{C}^{1}$ flow defined in a compact manifold $M$ and 
$\Lambda \subset M$ be a compact $\varphi$-invariant subset such that all singularities of $\varphi$ in $\Lambda$ are hyperbolic. Then, either $\varphi\mid_\Lambda$ has no regular periodic orbits or
$$\varepsilon_0(\varphi_{|\Lambda}) :=  
\inf\{T > 0:  T \, \text{is period of a regular periodic orbit  of}\; \varphi\mid_\Lambda\} 
> 0.$$
\end{lemma}

\begin{proof} 
Assume that $\varphi\mid_\Lambda$ has regular periodic orbits.
Since all singularities are hyperbolic and $\Lambda$ is compact then these must be finite in number, say $\sigma_1, \cdots, \sigma_n$.
For all $1 \leq i \leq n$ let $V_i$ be a small neighborhood of $\sigma_i$ given by Hartman-Grobman theorem (see \cite{PalisMelo}, p. 68). Every small $V_i$ which is a neighborhood of a sink or source contains no regular periodic orbits. On other hand, if a periodic orbit intersects a neighborhood $V_i$ associated with a hyperbolic singularity of saddle type 
then its period is bounded below by a uniform constant (which is inversely proportional to the largest eigenvalue
of the unstable bundle among the hyperbolic saddles). 
It remains to prove that all periodic orbits in 
$S = \Lambda \cap (M \setminus \bigcup_{i = 0}^n V_i)$
have a period bounded away from zero.  By construction $S$ is a compact set without singularities.
For every $x\in S$ let $\delta_x>0$ and $B_x=U^{\delta_x}_x$ be a tubular neighborhood associated to $x$. 
By compactness of $S$ the open covering $(B_x)_{x\in S}$ admits a finite cover $(B_{x_j})_{j=1}^{\kappa}$.
It is now clear that any periodic orbit in $S$ has period larger than $\min_{1\le j \le \kappa} \delta_{x_j}$.
Thus $\varepsilon_0(\varphi_{|\Lambda}) > 0$, 
which finishes the proof of the lemma.
\end{proof}

\subsection{Expansiveness and local triviality at regular points}

In the next lemma, we prove the existence, uniqueness and continuity of a local reparameterization for an element in the centralizer of an expansive flow $\varphi$. 
In the case that $\varphi$ has no regular period orbits on $\Lambda$, set for simplicity $\vep_0(\varphi\mid_\Lambda)=+\infty$.
For brevity we will write $\vep_0(\varphi)$ in place of $\vep_0(\varphi\mid_\Lambda)$.

\begin{lemma}\label{singepsilon1} Let $\varphi$ a $\mathcal{C}^{1}$ flow on a compact manifold $M$
and $\Lambda \subset M$ be a compact $\varphi$-invariant subset such that the restriction $\varphi_{|\Lambda}$ is expansive and all the singularities of $\varphi$ in $\Lambda$ are hyperbolic.
If $\psi=(\psi_s)_{s \in \mathbb{R}}$ belongs to $\mathcal{Z}^{1}(\varphi_{|\Lambda})$ 
then for any $0 < \varepsilon < \varepsilon_0(\varphi_{|\Lambda})/3$ there exists $\mu > 0$ and a unique function $z :\, [-\mu, \, \mu] \times \left(\Lambda \setminus Sing(\varphi_{|\Lambda})\right) \rightarrow (-\varepsilon, \, \varepsilon)$ such that $\psi_s(x) = \varphi_{z(s, \, x)}(x)$ for any $(s, \, x) \in [-\mu, \, \mu] \times \left(\Lambda \setminus Sing(\varphi_{|\Lambda})\right)$. Moreover, 
\begin{itemize}
\item[(i)\; ] $z$ is continuous,
\item[(ii) ] if $t$, $s$, $t + s \in [-\mu, \, \mu]$ then $z(t + s, \, x) = z(t, \, x) + z(s, \, \psi(t, \, x))$.
\end{itemize}
\end{lemma}

\begin{proof} 
By Lemma \ref{e0} we have that $\varepsilon_0(\varphi_{|\Lambda}) > 0$. Given $0 < \varepsilon 
< \varepsilon_0(\varphi)/3$, let $\delta > 0$ be given by the expansiveness property (recall Definition~\ref{expansiveflow}).
Since $\Lambda$ is compact and $\varphi$-invariant, there exists $\mu > 0$ such that 
$\sup\limits_{|s| \leq \mu}\left\{d(Id, \, \psi_s)\right\} < \delta$ and, consequently, 
\begin{align*}
d(\varphi_t(x), \, \varphi_t(\psi_s(x)))  = d(\varphi_t(\psi_0(x)), \, \varphi_t(\psi_s(x))) 
                                      = d(\psi_0(\varphi_t(x)), \, \psi_s(\varphi_t(x))) 
                                       < \delta
\end{align*}
for every $x \in \Lambda$, $|s| \leq \mu$ and $t \in \mathbb{R}$. This implies (taking $h(t) = t$ in Definition \ref{expansiveflow}) that there exists $t_0 \in \mathbb{R}$ such that $\varphi_{t_0}(\psi_s(x)) = \varphi_{t_0 + \eta}(x)$ for some $\eta \in (-\varepsilon, \, \varepsilon)$. Consequently $\psi_s(x) = \varphi_{\eta}(x)$
belongs to the orbit of $x$ relative to the flow $(\varphi_t)_t$.

This defines uniquely a map $z :\, [-\mu, \, \mu] \times \left(\Lambda \setminus Sing(\varphi\mid_\Lambda)\right) \rightarrow (-\varepsilon, \, \varepsilon)$ such that $\psi_s(x) = \varphi_{z(s, \, x)}(x)$ for any $(s, \, x) \in [-\mu, \, \mu] \times \left(\Lambda \setminus Sing(\varphi\mid_\Lambda)\right)$.
In fact, if $z_1, \, z_2 :\, [-\mu, \, \mu] \times \left(\Lambda \setminus Sing(\varphi\mid_\Lambda)\right) \rightarrow (-\varepsilon, \, \varepsilon)$ are such that $\varphi_{z_1(s, \, x)}(x) = \psi_s(x) = \varphi_{z_2(s, \, x)}(x)$ then $\varphi_{z_1(s, \, x) - z_2(s, \, x)}(x) = x$ with $|z_1(s, \, x) - z_2(s, \, x)| \leq |z_1(s, \, x)| + |z_2(s, \, x)| < 2/3 \, \varepsilon_0(\varphi)$. So, by definition of $\varepsilon_0(\varphi)$ we conclude that $z_1(s, \, x) = z_2(s, \, x)$.

To prove (i) assume by contradiction that $z$ is not continuous. Then there are $\delta_0>0$, $(s, \, x) \in [-\mu, \, \mu] \times \left(\Lambda \setminus Sing(\varphi\mid_\Lambda)\right)$ and a sequence $(s_n, \, x_n)_{n \in \mathbb{N}}$ 
in $[-\mu, \, \mu] \times \left(\Lambda \setminus Sing(\varphi\mid_\Lambda)\right)$ that converges to $(s,x)$ and such that 
$|z(s_{n}, \, x_{n}) - z(s, \, x)| \geq \delta_0$ for all $n \in \mathbb{N}$.
This implies there exists $\delta_1 > 0$ such that 
$d(\varphi_{z(s_n, \, x_n)}(x), \; \varphi_{z(s, \, x)}(x)) \ge \delta_1$ 
for all $n \in \mathbb{N}$. On the other hand
\begin{align*}
\delta_1 & \leq d(\varphi_{z(s_n, \, x_n)}(x), \; \varphi_{z(s, \, x)}(x)) \\
         & \leq d(\varphi_{z(s_n, \, x_n)}(x), \; \varphi_{z(s_n, \, x_n)}(x_n)) +  d(\varphi_{z(s_n, \, x_n)}(x_n), \; \varphi_{z(s, \, x)}(x))\\
         & =    d(\varphi_{z(s_n, \, x_n)}(x), \; \varphi_{z(s_n, \, x_n)}(x_n)) +   d(\psi_{s_n}(x_n), \; \psi_s(x)).
\end{align*}
Using further that $(z(s_n, \, x_n))_{k \in \mathbb{N}}$ is bounded then, up to consider a subsequence, we may assume without less of generality that $z(s_n, \, x_n) \rightarrow t_0$ as $n \rightarrow \infty$. In consequence, the right hand side above tends to zero by continuity of the flow $\varphi$ and of the time $t_0$-map $\varphi_{t_0}$, contradicting 
the existence of $\delta_1 > 0$. This proves the continuity claimed in (i).
To prove the property (ii), by the equality
\begin{align*}
\varphi_{z(t + s, \, x)}(x) & = \psi_{t + s}(x) 
                          = \psi_t(\psi_s(x)) \\
                         & = \varphi_{z(t, \, \psi_s(x))}(\psi_s(x)) \\
                         & = \varphi_{z(t, \, \psi_s(x) )}(\varphi_{z(s, \, x)}(x)) \\
                         & = \varphi(z(t, \, \psi_s(x) ) + z(s, \, x),\, x)
\end{align*}
and uniqueness of local reparameterization $z$, for $0 < \varepsilon < \varepsilon_0(\varphi)/3$ we conclude that $z(t + s, \, x) = z(t, \, \psi(s, \, x)) + z(s, \, x)$ for any $t, \, s, \, t + s \in [-\mu, \, \mu]$ and every regular point $x \in \Lambda$.
\end{proof}

\subsection{Unique continuous extension for the local reparameterization}\label{sec:extend1}

In what follows we construct an extension of the continuous reparameterization described in Lemma \ref{singepsilon1}
to $\mathbb{R} \times \left(\Lambda \setminus Sing(\varphi\mid_\Lambda)\right)$ . More precisely:

\begin{proposition}\label{extencao}  Let
 $\varphi$ be a continuous flow  in $M$ and $\psi$ be a continuous flow such that there exists $\mu > 0$ and
  $z :\, [-\mu, \, \mu] \times  \left(\Lambda \setminus Sing(\varphi\mid_\Lambda)\right) \rightarrow (-\varepsilon, \, \varepsilon)$ such that $\psi(s, \, x) = \varphi(z(s, \, x), \, x)$ for any $(s, \, x) \in [-\mu, \, \mu]  \times \left(\Lambda \setminus Sing(\varphi\mid_\Lambda)\right)$, where $0 < \varepsilon < \varepsilon_0(\varphi)/3$.
There exists a unique continuous function $p :\, \mathbb{R}  \times \left(\Lambda \setminus Sing(\varphi\mid_\Lambda)\right) \rightarrow \mathbb{R}$ which extends $z$ and satisfies $\psi(s, \, x) = \varphi(p(s, \, x), \, x)$ for any $(s, \, x) \in \mathbb{R}  \times \left(\Lambda \setminus Sing(\varphi\mid_\Lambda)\right)$.
\end{proposition}

\begin{proof}
We first prove the existence of the reparameterization.
Take $N \ge 1$ so that $2^{-N} < \mu$. Now, let 
$z_1: \, [1/2^N, \, 2/2^N] \times \left(\Lambda \setminus Sing(\varphi\mid_\Lambda)\right) \rightarrow \mathbb{R}$ 
be the continuous function given by 
$$
z_1(t, \, x) = z\Big(t - \dfrac{1}{2^N}, \, x\Big) + z\Big(\dfrac{1}{2^N}, \psi\Big(t - \dfrac{1}{2^N}, \, x\Big)\Big).
$$ 
A simple computation shows that  $z_1(\frac{1}{2^N}, \, x) = z(\frac{1}{2^N}, \, x)$
for any point $x \in \Lambda \setminus Sing(\varphi\mid_\Lambda)$. This means that the functions
coincide in the extreme point of the interval $[0,1/2^N]$.
Now we claim that $\psi(t, \, x) = \varphi(z_1(t, \, x), \, x)$ for every $(t, \, x) \in [1/2^N, \, 2/2^N]  \times \left(\Lambda \setminus Sing(\varphi\mid_\Lambda)\right)$.
Fix $(t, \, x) \in [1/2^N, \, 2/2^N]  \times \left(\Lambda \setminus Sing(\varphi\mid_\Lambda)\right)$. On the one hand, by Lemma \ref{singepsilon1},  
$$
0 = z\Big(\frac{1}{2^N} - \frac{1}{2^N}, \psi(t, \, x)\Big) = z\Big(\frac{1}{2^N}, \psi\Big(t - \frac{1}{2^N}, \, x\Big)\Big) + z\Big(-\frac{1}{2^N}, \psi(t, \, x)\Big)
$$
and consequently, $z(-1/2^N, \psi(t, \, x)) = -z(1/2^N, \psi(t - 1/2^N, \, x))$ for every $(t, \, x) \in [1/2^N, \, 2/2^N] \times \left(\Lambda \setminus Sing(\varphi\mid_\Lambda)\right)$. On the other hand, since $|z(\cdot)|$ is bounded above by $\vep_0(\varphi\mid_\Lambda)/3$,
\begin{align*}
\varphi\Big(-z\Big(\frac{1}{2^N}, \psi&\Big(t - \frac{1}{2^N}, \, x\Big)\Big), \varphi(z_1(x, \, t), \, x)\Big) \\
		& =  \varphi\Big(-z\Big(\frac{1}{2^N}, \psi\Big(t - \frac{1}{2^N}, \, x\Big)\Big) + z_1(x, \, t), \, x\Big) \\
                   & =  \varphi\Big(z\Big(t - \frac{1}{2^N}, \, x\Big), \, x\Big) \\
                   & =  \psi\Big(t - \frac{1}{2^N}, \, x\Big) \\
                   & =  \psi\Big(-\frac{1}{2^N}, \, \psi(t, \, x)\Big) \\
                   & =  \varphi\Big(z\Big(-\frac{1}{2^N}, \, \psi(t, \, x)\Big), \, \psi(t, \, x)\Big) \\
                   & =  \varphi\Big(-z\Big(\frac{1}{2^N}, \psi\Big(t - \frac{1}{2^N}, \, x\Big)\Big), \, \psi(t, \, x)\Big)
\end{align*}
from which the claim follows.
This proves our claim and so $\psi$ is a local reparameterization of the flow $\varphi$ on the interval $[1/2^N, \, 2/2^N]$. 
Inductively, for each integer $k \geq 1$ let
$z_k: \, \Big[\dfrac{k}{2^N}, \, \dfrac{k+1}{2^N}\Big] \times \left(\Lambda \setminus Sing(\varphi\mid_\Lambda)\right) \rightarrow \mathbb{R}$ be the continuous function given by
$$
z_k(t, \, x) = z\Big(t - \dfrac{k}{2^N}, \, x\Big) + \sum\limits_{i=1}^k z\Big(\dfrac{1}{2^N}, \psi\Big(t - \dfrac{i}{2^N}, \, x\Big)\Big).
$$
For any $x \in \Lambda \setminus Sing(\varphi)$ we observe that 
$z_k\Big(\frac{k+1}{2^N}, \, x\Big) = z_{k+1}\Big(\frac{k+1}{2^N}, \, x\Big)$
because
\begin{align*}
z_k\Big(\frac{k+1}{2^N}, \, x\Big) & = z\Big(\frac{1}{2^N}, \, x\Big) + \sum\limits_{i=1}^k z\Big(\frac{1}{2^N}, 
\psi\Big(\frac{k+1-i}{2^N}, \, x\Big)\Big) \\
                     & = z\Big(\frac{1}{2^N}, \, x\Big) + \sum\limits_{j=1}^k z\Big(\frac{1}{2^N}, \psi\Big(\frac{j}{2^N}, \, x\Big)\Big) 
                      = \sum\limits_{j=0}^k z\Big(\frac{1}{2^N}, \psi\Big(\frac{j}{2^N}, \, x\Big)\Big) \\
                      & = z_{k+1}\Big(\frac{k+1}{2^N}, \, x\Big).
\end{align*}
In addition, $z_k$ satisfies the recursive expression
\begin{equation}\label{f)}
z_k(t, \, x) = z_{k-1}(t - 1/2^N, \, x) + z(1/2^N, \, \psi(t - 1/2^N, \, x)).
\end{equation}
Indeed, simple computations yield
\begin{align*}
z_k(t, \, x) & = z\Big(t - \frac{k}{2^N}, \, x\Big) + \sum\limits_{i=1}^k z\Big(\frac{1}{2^N}, \psi\Big(t - \frac{i}{2^N}, \, x\Big)\Big)\\
	    & = z\Big((t- \frac{1}{2^N}) - \frac{k-1}{2^N}, \, x\Big) 
	    + \sum\limits_{i=1}^k z\Big(\frac{1}{2^N}, \psi\Big((t- \frac{1}{2^N}) - \frac{i-1}{2^N}, \, x\Big)\Big)\\
	    & = z\Big((t- \frac{1}{2^N}) - \frac{k-1}{2^N}, \, x\Big) 
	    + \sum\limits_{i=1}^{k-1} z\Big(\frac{1}{2^N}, \psi\Big((t- \frac{1}{2^N}) - \frac{i}{2^N}, \, x\Big)\Big) \\
	   &  + z\Big(\frac{1}{2^N}, \psi\Big(t- \frac{1}{2^N}, \, x\Big)\Big)
\end{align*}
and proves the equality in \eqref{f)}. 
We need the following: \smallskip

\noindent {\bf Claim:} \textit{For every $x \in \Lambda \setminus Sing(\varphi\mid_\Lambda)$ and $t \in \Big[\dfrac{k}{2^N}, \, \dfrac{k+1}{2^N}\Big]$ the function $z_k$ satisfies $\psi(t, \, x) = \varphi(z_k(t, \, x), \, x).$}

\begin{proof}[Proof of the claim] 
The claim for $k = 1$ has already been confirmed.
By induction, assume that the affirmation is true for $k-1$. Then, using (\ref{f)}), we obtain that
\begin{align*}
\varphi\big(-z\big(\frac{1}{2^N}, \,  \psi\big(t - \frac{1}{2^N}, \, x\big)\big), \, \varphi(z_k(t, \, x), \, x)\big)  
									& = \varphi(-z(\frac{1}{2^N}, \, \psi(t - \frac{1}{2^N}, \, x)) + z_k(t, \, x), \, x)\\
                                                                             & = \varphi\big(z_{k-1}\big(t - \frac{1}{2^N}, \, x\big), \, x\big)\\
                                                                             & = \psi\big(t - \frac{1}{2^N}, \, x\big)
                                                                              = \psi\big(-\frac{1}{2^N}, \, \psi(t, \, x)\big)\\
                                                                             & = \varphi\big(z\big(-\frac{1}{2^N}, \, \psi(t, \, x)\big), \, \psi(t, \, x)\big)\\
                                                                             & = \varphi\big(-z\big(\frac{1}{2^N}, \, \psi\big(t - \frac{1}{2^N}, \, x\big)\big), \, \psi(t, \, x)\big).\\
\end{align*}
Since the time-$s$ map $\varphi_s$ is a diffeomorphism for all $s \in \mathbb{R}$, we conclude that $\varphi(t, \, x) = \varphi(z_k(t, \, x), \, x)$ for every $x \in \Lambda \setminus Sing(\varphi\mid_\Lambda)$ and $t \in \Big[\dfrac{k}{2^N}, \, \dfrac{k+1}{2^N}\Big]$, which proves the claim.
\end{proof}

Clearly, a completely similar argument as above is enough to extend $z(t, \, x)$ for all negative $t$. Just
consider the continuous function $\overline{z}_1: \, [-2/2^N, \, -1/2^N] \times \left(\Lambda \setminus Sing(\varphi\mid_\Lambda)\right) \rightarrow \mathbb{R}$ given by
$$\overline{z}_1(t, \, x) = z\Big(t + \dfrac{1}{2^N}, \, x\Big) + z\Big(-\dfrac{1}{2^N}, \varphi\Big(t + \dfrac{1}{2^N}, \, x\Big)\Big).$$
Computations similar to the ones above yield that for any $x \in \Lambda \setminus Sing(\varphi\mid_\Lambda)$, the function  $\overline{z}_1(\cdot, \, x)$ satisfies $\overline{z}_1(-1/2^N, \, x) = z(-1/2^N, \, x)$ and $\varphi(t, \, x) = \varphi(\overline{z}_1(t, \, x), \, x).$
Then, for each positive integer $k$, take $\overline{z}_k: \, \left[-(k+1)/2^N, \, -k/2^N\right]  \times \left(\Lambda \setminus Sing(\varphi\mid_\Lambda)\right) \rightarrow \mathbb{R}$ given by
$$\overline{z}_k(t, \, x) = z\Big(t + \dfrac{k}{2^N}, \, x\Big) + \sum\limits_{i=1}^k z\Big(-\dfrac{1}{2^N}, \varphi\Big(t + \dfrac{i}{2^N}, \, x\Big)\Big),$$
which satisfies
$\varphi(t, \, x) = \varphi(\overline{z}_k(t, \, x), \, x)$
for every $t \in \left[-(k+1)/2^N, \, -k/2^N\right]$ and $x \in \mathbb{R} \times \left(\Lambda \setminus Sing(\varphi\mid_\Lambda)\right)$.
Altogether we get a well-defined continuous function $p : \mathbb{R}  \times \left(\Lambda \setminus Sing(\varphi\mid_\Lambda)\right) \rightarrow \mathbb{R}$ given by
$$p(t, \, x) = \left\{ \begin{array}{cl}
z(t, \, x),              & \mbox{if } t \in [-1/2^N, \, 1/2^N] \\
z_k(t, \, x),            & \mbox{if } t \in \left[k/2^N, \, (k+1)/2^N\right],\;  k \in \mathbb{N}\\
\overline{z}_k(t, \, x), & \mbox{if } t \in \left[-(k+1)/2^N, \, -k/2^N\right] ,\;  k \in \mathbb{N}.
\end{array} \right.
$$
By construction, for every $\psi \in \mathcal{Z}^{\infty}(\varphi)$ there exists a continuous $p$ such that 
$\psi(t, \, x) = \varphi(p(t, \, x), \, x)$ for every $(t, \, x) \in \mathbb{R} \times \left(\Lambda \setminus Sing(\varphi\mid_\Lambda)\right)$. 
This concludes the proof of the existence of the reparameterization. 

\medskip

In the remaining of the proof of the proposition we are left to prove the uniqueness of the reparameterization among non-singular points. For this, suppose there are $p_1, \, p_2 :\, \mathbb{R}  \times \left(\Lambda \setminus Sing(\varphi\mid_\Lambda)\right) \rightarrow \mathbb{R}$, continuous extensions of $z$ such that $\varphi(p_1(t, \, x), \, x) = \psi(t, \, x) = \varphi(p_2(t, \, x), \, x)$ for all $(t, \, x) \in \mathbb{R}  \times \left(\Lambda \setminus Sing(\varphi\mid_\Lambda)\right)$. Fix $x \in \Lambda \setminus Sing(\varphi\mid_\Lambda)$ and consider the continuous function $\alpha_x(t) = p_1(t, \, x) - p_2(t, \, x)$.
By Lemma \ref{singepsilon1}, $\alpha_x^{-1}(0) \supset [-\mu, \, \mu]$ (thus $\alpha_x^{-1}(0)$ is non empty). Moreover, 
$\alpha_x^{-1}(0)$ is closed by the continuity of $\alpha_x$. 

Assume, by contradiction, that $\alpha_x^{-1}(0)\neq \mathbb{R}$, Then, as $\alpha_x$ is continuous, there exists $t_0 = \max\{t > 0: [0, \, t] \subset \alpha_x^{-1}(0)\} \geq \mu$ or $\min\{t < 0: [t, \, 0] \subset \alpha_x^{-1}(0)\} \leq -\mu$. We assume the first case holds (the second is completely analogous).
By continuity of the reparameterizations $p_1(t_0, \, x) = p_2(t_0, \, x)$. Moreover, if $t \in [-\mu, \, \mu]$, then 
$\varphi(p_i(t + t_0, \, x), \, x) = \psi(t + t_0, \, x) = \psi(t, \, \psi(t_0, \, x))$ for $i \in \{1, \, 2\}$.

Define $x' = \varphi(p_i(t_0,x),x)$ (which is the same point for $i = 1,2$). Since for each $i = 1,2$ there holds 
$\psi(t,x) =\varphi(p_i(t,x),x))$ for all $(t, x) \in \mathbb R \times (\Lambda \setminus Sing(\varphi \mid_\Lambda))$, then $x' = \psi(t_0, x)$. 
By Lemma \ref{singepsilon1} we know the existence of a unique function $z :\, [-\mu, \, \mu] \times \left(\Lambda \setminus Sing(\varphi\mid_\Lambda)\right) \rightarrow (-\varepsilon, \, \varepsilon)$ such that $\psi(s, \, x) = \varphi(z(s, \, x), \, x)$ for any $(s, \, x) \in [-\mu, \, \mu] \times \left(\Lambda \setminus Sing(\varphi\mid_\Lambda)\right)$.
Then for $t \in [-\mu, \mu]$ there holds $\psi(t, x') = \varphi(z(t, x'), x')$.
Since $\psi(t, x') = \psi(t, \psi(t_0, x)) = \psi(t + t_0, x) = \varphi(p_i(t + t_0, x), x)$, there holds
$$
\varphi(z(t, x'), x') = \varphi(p_i(t + t_0, x), x).
$$
This implies, by the uniqueness of $z$, that $p_1(t + t_0, x) = p_2(t + t_0, x)$ for all $t\in [-\mu,\mu]$
contradicting the maximality of $t_0$.
This completes the proof of Proposition~\ref{extencao}.
\end{proof}

\subsection{Invariance of reparameterizations along regular orbits}\label{sec:invarian1}

In what follows we prove that the unique reparameterization obtained in Proposition~\ref{extencao} is invariant along orbits of regular points.

\begin{lemma}\label{lemainvorbitas1}
If $p$ is the reparameterization given in Proposition~\ref{extencao} then $p(t, \, x) = p(t, \, \varphi(s, \, x))$ for every $t,s \in \mathbb{R}$
and any $x \in \Lambda \setminus Sing(\varphi\mid_\Lambda)$.
Moreover, there exists a unique continuous function $A : \Lambda \setminus Sing(\varphi\mid_\Lambda)\rightarrow \mathbb{R}$ so that 
$A(\varphi(t, \, x))=A(x)$ and
$p(t, \, x) = A(x) t$ for every $t \in \mathbb{R}$ and $x \in \Lambda \setminus Sing(\varphi\mid_\Lambda)$.
\end{lemma}

\begin{proof} Initially we observe that, since $\psi$ commutes with $\varphi$,
\begin{align*}
\varphi(s + p(t, \, x), \, x) & = \varphi(s, \, \varphi(p(t, \, x), \,x)) 
                               = \varphi(s, \, \psi(t, \, x)) \\
                              & = \psi(t, \, \varphi(s, \, x)) 
                               = \varphi(p(t, \, \varphi(s, \, x)), \, \varphi(s, \, x)) \\
                              & = \varphi(p(t, \, \varphi(s, \, x)) + s, \, x).
\end{align*}
Therefore, for $\mu$ sufficiently small and $t \in [-\mu, \mu]$ this equality implies that 
\begin{equation}\label{eq:Ainv}
p(t, \, x) = p(t, \, \varphi(s, \, x))
\end{equation}
for every $s\in\mathbb R$.
From the construction and uniqueness of function $p$ together with the recursive expression (\ref{f)})
it follows 
that $p(t, \, x) = p(t, \, \varphi(s, \, x))$ for all $t, s \in \mathbb{R}$. Using that 
$p(t, \, x) = p(t, \, \varphi(s, \, x))$ for every $s\in\mathbb R$ together with Proposition~\ref{extencao},
\begin{align*}
\varphi(p(t  + s, \, x), \, x) & = \psi(t + s, \, x) 
                                = \psi(t, \, \psi(s, \, x)) \\
                               & = \varphi(p(t, \, \psi(s, \, x)), \, \psi(s, \, x)) \\
                               & = \varphi(p(t, \, \varphi(p(s, \, x), \, x)), \, \varphi(p(s, \, x), \, x)) \\
                               & = \varphi(p(t, \, \varphi(p(s, \, x), \, x)) + p(s, \, x), \, x) \\
                               & = \varphi(p(t, \, x) + p(s, \, x), \, x)
\end{align*}
for all $t \in \mathbb{R}$. The uniqueness of $p$ implies that $p(t  + s, \, x) = p(t, \, x) + p(s, \, x)$ for all $t,s$.
Since $p(\cdot, x)$ is continuous then it is linear. Hence there exists a continuous map $A: M \rightarrow \mathbb{R}$ 
so that $p(t, \, x) = A(x) t$ for all $x \in \Lambda \setminus Sing(\varphi\mid_\Lambda)$. Together with \eqref{eq:Ainv},
this proves that $A(\varphi(t, \, x))=A(x)$ for every $t \in \mathbb{R}$ and $x \in \Lambda \setminus Sing(\varphi\mid_\Lambda)$
and finishes the proof of the lemma.
\end{proof}

\subsection{Extension of the reparameterization to singular points}\label{sec:extension}

Under the assumption that all singularities are hyperbolic and satisfy non-ressonance conditions, we will prove that the reparameterization obtained 
in Proposition~\ref{extencao} can be continuously extended to the singular points.
For this, initially we deduce the following version of Kopell's theorem (\cite{Ko70}, Theorem 6) for linear contractions.
For any complex number $\lambda\in\mathbb C$ let $\emph{Re}(\lambda)$ denote its real part.

\begin{lemma}\label{kopellflows} 
Given $B\in GL(n,\mathbb R)$ let $\varphi = (e^{t \cdot B})_{t\in\mathbb R}$ be such that $0$ is a sink and assume that
it has non-resonant eigenvalues.  If $\lambda_1, \, \cdots, \, \lambda_n$ are the eigenvalues  of $B$ and $m$ is the least 
positive integer such that 
\begin{equation}\label{eq:K}
m \cdot \Big(\max\limits_{1 \leq i \leq n} \emph{Re}(\lambda_i) \Big) < \min\limits_{1 \leq i \leq n} \emph{Re}(\lambda_i)
\end{equation}
then $\mathcal{Z}^m(\varphi)$ is the set of linear flows $(e^{s \cdot C})_{s \in \mathbb{R}}$ where $C\in GL(n,\mathbb R)$ is such that $B \cdot C = C \cdot B$.
\end{lemma}

\begin{proof} 
As $0$ is a sink for $B$ then $e^B$ is a linear contraction and, 
the fact that eigenvalues of $B$ satisfy condition ~\eqref{eq:K} implies that the non-ressonance
conditions in \cite[Theorem~6]{Ko70} hold. 
Since $0$ is a sink then all eigenvalues $\lambda_1, \, \cdots, \, \lambda_n$ of $B$
have real negative part and $| e^{\lambda_i} |< 1$ for all $1\le i \le n$.
Thus, if $m$ is given by \eqref{eq:K} above, the result of Kopell implies that the $\cC^m$-centralizer
of the linear automorphism $e^B$ acting on $\mathbb R^n$ 
is constituted exclusively by linear maps. So, any element of $\mathcal{Z}^m(\varphi)$ is a flow 
$\psi=(\psi_s)_{s \in \mathbb{R}}$ such that $\psi_s$ is a linear map for all $s \in \mathbb{R}$ (since all of these should commute with the time-$1$ map $e^B)$.

In what follows we show that any flow of linear maps $\psi\in \mathcal{Z}^m(\varphi)$ is of exponential type. 
We claim that $\psi_s=e^{sC}$ for every $s\in \mathbb R$, where
\begin{equation*}
C (x) := \lim\limits_{s \rightarrow 0} \frac{\psi_s(x) - x}{s}
	= \frac{\partial \psi_s}{\partial s} \mid_{s=0} (x)
\end{equation*}
for $x\in \mathbb R^n$. Let $T > 0$ be fixed.  Since the maps $t \rightarrow \|e^{t C}\|$ and $s \rightarrow \|\psi_s\|$ are continuous 
there exists a constant $K > 0$  (depending on $T$) such that $\|\psi_s\| \cdot \|e ^{t C}\| \leq K$ for all $0 \leq t, \, s \leq T$.
Given $\varepsilon > 0$, as $C = \lim\limits_{s \rightarrow 0} \frac{e^{s C} - Id}{s}$,  one can choose $0<\delta \le T$ 
such that 
$$
\|\frac{\psi_h - e^{h C}}{h}\| < \frac{\varepsilon}{KT} \quad\text{for all}\quad 0 \leq h \leq \delta.
$$
Now pick $n \in \mathbb{N}$ such that $\frac{T}{n} < \delta$. Then, using the triangular inequality, for any $t \in [0, \, T]$
\begin{align*}
\|\psi_t - e^{t C}\| 
	& = \|\psi_{n \cdot \frac{t}{n}} - e^{n \frac{t}{n} C}\|  \\
	& \leq \sum\limits_{k=0}^{n-1} \|\psi_{(n-k) \cdot \frac{t}{n}} \circ e^{k \frac{t}{n} C} - \psi_{(n-k-1) \cdot \frac{t}{n}} \circ e^{(k+1) \frac{t}{n} C}\| \\
	& \leq \sum\limits_{k=0}^{n-1} \|\psi_{(n-k-1) \cdot \frac{t}{n}}\| \cdot \|\psi_{\frac{t}{n}} - e^{\frac{t}{n} C}\| 
		\cdot \|e^{k \frac{t}{n} C}\| \\
	& \leq  K n \cdot \frac{\varepsilon}{K T } \cdot \frac{t}{n} \le \varepsilon.
\end{align*}
Since $\vep$ was chosen arbitrary the later proves that $\psi_t = e^{t C}$ for all $t \in [0, \, T]$. 
The group property implies that the equality holds for all $t \in \mathbb{R}$. Finally, such flows commute if and only
if $BC=CB$. This completes the proof of the lemma.
\end{proof}

Next, we use Lemma~\ref{kopellflows} to prove that the reparameterization can be continuously extended to 
the singularities $\sigma_i$. Indeed, we will prove that the restriction of the local reparameterizations to the stable and  unstable manifolds  $W^s(\sigma_i) \setminus \sigma_i$ and $W^u(\sigma_i) \setminus \sigma_i$ of a non-resonant hyperbolic singularity $\sigma_i$ of a $C^\infty$-flow $\varphi$ are necessarily constant.
This, together with the fact that 
$$
\mathcal{Z}^{\infty}(\varphi) \subset \mathcal{Z}^m(\varphi) \subset \mathcal{Z}^1(\varphi)
$$ 
for every $m \in \mathbb{N}$ and every $C^\infty$-flow $\varphi$, will be a key step in the proof of the theorem.

\begin{lemma}\label{extencaosing} 
Let $\varphi: \mathbb{R} \times \Lambda \rightarrow \Lambda$ a $C^\infty$-expansive flow defined on a compact Riemannian manifold $M$. Suppose that the singularities $\sigma_1, \, \cdots, \, \sigma_k$ of $\varphi$ are hyperbolic, set $B_i=\frac{d}{dt}\varphi(t, \, \sigma_i)_{|t = 0}$ and let $T_{\sigma_i} M= E_ i^s \oplus E_ i^u $ be the hyperbolic splitting, $1 \leq i \leq k$. 
If the eigenvalues of $B_i \mid_{E^*_i}$ are distinct and non resonant for every $1 \leq i \leq k$ and $*\in\{s,u\}$
then the continuous function $A(\cdot)$ given by Lemma~\ref{lemainvorbitas1} admits a continuous extension to $\Lambda$.
\end{lemma}

\begin{proof}
Since the singularities of $\varphi$ are hyperbolic then these are isolated and, for that reason, it is enough to extend
the function $A(\cdot)$ to each singularity recursively. 
We subdivide the proof in two cases, corresponding to the case where the singularities are either sinks/sources
or saddles. 

\medskip

\noindent {\emph{Case 1}}: $\sigma_i$ is a sink or a source.

\smallskip
\noindent Assume without loss of generality that $\sigma_i$ is a sink. Indeed, in the case that $\sigma_i$ is a source 
the proof is completely analogous just by considering the time reversed flow $(\varphi_{-t})_{t\in\mathbb R}$.
Since the eigenvalues of $B_i = \frac{d}{dt}\varphi(t, \, \sigma_i)|_{t = 0}$ are non-resonant, by Sternberg linearization theorem (see~\cite{Stern1}) there exists a neighborhood $W_i$ of $\sigma_i$ such that 
the flow $(\varphi_t)_t$ is $C^\infty$-linearizable in $W_i$: there exists a $C^\infty$-chart $\zeta_i$
on an open set in $\mathbb R^{\dim M}$ so that $\eta_t:=\zeta_i\circ \varphi_t \circ \zeta_i^{-1}$, $t\in \mathbb R$, defines a linear flow on $W_i$.
This conjugation $\zeta_i$ induces a natural isomorphism between $\mathcal{Z}^\infty((\varphi_t)_{t\in\mathbb{R}})$ and $\mathcal{Z}^\infty((\eta_s)_{s\in\mathbb{R}})$ in the sense that $(\psi_s)_{{s\in\mathbb{R}}} \in 
\mathcal{Z}^\infty((\varphi_t)_{t\in\mathbb{R}})$ if and only if $(h \circ \psi_s \circ h^{-1})_{{s\in\mathbb{R}}} \in \mathcal{Z}^\infty((\eta_t)_{t\in\mathbb{R}})$.

We use this fact to determine the centralizer of $\varphi$ in a neighborhood of the singularities $\sigma_i$.
On the one hand, Lemma \ref{kopellflows} implies that the centralizer 
$\mathcal{Z}^{\infty}((e^{s B_i})_{s\in \mathbb{R}})$ is formed by the linear flows
$(e^{s C})_{s\in \mathbb{R}}$ where the linear map $C$ satisfies $B_i \cdot C = C \cdot B_i$.
On the other hand, it follows from Lemmas \ref{singepsilon1} and~\ref{lemainvorbitas1} and Proposition~\ref{extencao} 
that any  $\psi \in \mathcal{Z}^{\infty}(\varphi)$ is reparameterization of $\varphi$, meaning that there exists continuous function $A: W_i \setminus\{\sigma_i\}
\to \mathbb R$ so that $\psi_t(x) = \varphi_{A(x) \cdot t} ( x)$ for all $t\in \mathbb R$.
Thus we are interested in determining which linear flows $(e^{s C})_{s\in \mathbb R}$ preserve the orbits of 
the linear flow $(e^{t B_i})_{t\in\mathbb R}$. 
We claim that all such flows are of the form $(e^{s C})_{s\in \mathbb R}$ with $C = c \, B_i$, for some $c \in \mathbb{R}$.
If, for all $c \in \mathbb{R}$, $C \neq c \, B_i$  then there would exist $x \in \mathbb{R}^n$ such that the vectors 
$\{B_i (x), \; C (x) \}$ are linearly independent and, consequently, the orbits of $x$ by the two flows are 
transversal at $x$. This would contradict the fact that $(e^{s C} x)_{s\in\mathbb R}$ is a reparameterization of $(e^{t B_i} x)_{t\in\mathbb R}$ and proves the claim.
Finally we conclude via the conjugation given by Sternberg linearization theorem 
that, in the linearizing coordinates,  the function $A(x)$ given by Lemma \ref{lemainvorbitas1} is constant in 
$W^s(\sigma_i) \setminus \{\sigma_i\}$, hence it admits a continuous extension to $\sigma_i$.

\medskip

\noindent{\emph{Case 2:} $\sigma_i$ is a saddle.}
\smallskip

\noindent Let $W_i$ be a neighborhood of $\sigma_i$ given by Hartman-Grobman's theorem:
the flow $\varphi$ is topologically conjugate to the hyperbolic linear flow $(e^{t B_i})_{t\in \mathbb R}$ on $W_i$,
where $B_i = \frac{d}{dt}\varphi(t, \, \sigma_i)|_{t = 0}$.
 Choose a suitable base of $\mathbb R^d$ so that $B_i =$ diag$\{B_{1,i}, \, B_{2,i}\}$ and 
$B_{1,i}$ and $B_{2,i}^{-1}$ are contractions. Moreover, since the eigenvalues of $B$ are non-resonant,
we may reduce $W_i$ if necessary 
to guarantee that the flow $\varphi$ restricted to the open neighborhood $S_i := W_i \cap W^s(\sigma)$ of $p$ in $W^s(\sigma_i)$ (resp. open neighborhood $U_i := W_i \cap W^u(\sigma)$ of $p$ in $W^u(\sigma_i)$)
is $C^\infty$-conjugate to the linear contraction $(e^{t B_{1,i}})_{t\in\mathbb R}$ (resp. to the linear
expansion $(e^{t B_{2,i}})_{t\in\mathbb R}$).
Using Case 1 to deal, independently, with both linear flows we deduce that  the reparameterization $A(x)$ given by 
Lemma~\ref{lemainvorbitas1} is constant along both invariant manifolds $W^s(\sigma_i) \setminus \{\sigma_i\}$ and $W^u(\sigma_i) \setminus \{\sigma_i\}$.
To conclude that $A(x)$ extends continuously to $\sigma_i$ it is enough to show that, in the linearizing coordinates, 
there exists $c\in \mathbb R$ so that $A(x) = c$ is constant for all $x \in [W^s(\sigma_i) \cup W^u(\sigma_i)] \setminus \{\sigma_i\}$.
For this, consider a compact cross-section $\Sigma$ that is transversal to $W^s(\sigma_i)$, a compact section $\Sigma^{'}$ that it is transversal to $W^u(\sigma_i)$, a point $x \in \Sigma \cap W^s(\sigma_i)$ and a sequence $(x_n)$ of regular points in $\Sigma \setminus (W^s(\sigma_i) \cup W^u(\sigma_i))$ such that $x_n \rightarrow x$ when $n \rightarrow \infty$.
For all $n \geq 1$ large enough there exists a sequence $(t_n)$ in $\mathbb{R}$ such that 
$x_n'=\varphi(A(x_n) t_n, \, x_n) \in \Sigma^{'}$. By compactness of $\Sigma^{'}$ there exists a convergent subsequence 
$x_{n_k}' \to x' \in W^u(\sigma_i)$. Then, by the continuity of $A$ in $\mathbb{R}  \times \left(\Lambda \setminus
 Sing(\varphi\mid_\Lambda)\right)$ and its invariance along the orbits (cf. Proposition~\ref{extencao} and Lemma~\ref{lemainvorbitas1}) 
 imply that
\begin{align*}
A(x) = \lim\limits_{k \rightarrow \infty} A(x_{n_k}) & = \lim\limits_{k \rightarrow \infty} A(\varphi (A(x_{n_k}) t_{n_k}, (x_{n_k})) \\
                                                      & = A (\lim\limits_{k \rightarrow \infty} \varphi (A(x_{n_k}) t_{n_k}, x_{n_k})) 
                                                      = A(x^{'}).
\end{align*}
Thus $A$  extends continuously and uniquely to a reparameterization on $\Lambda$ so that 
$\psi_t(x) = \varphi_{A(x) t} (x)$ for all $(t, \, x) \in \mathbb{R} \times \Lambda$, $\psi \in \mathcal{Z}^{\infty}(\varphi)$. 
This completes the proof of the lemma.
\end{proof}

The later completes the proof of Theorem~\ref{thm:flows}.

\begin{remark}
In the case that the compact invariant set $\Lambda$ (cf. the statement of Lemma~\ref{extencaosing}) is a proper
subset of the manifold $M$ it is natural to address the continuity of the reparameterization. In the special case that 
$\Lambda$ coincides with the whole manifold, an argument using the implicit function theorem guarantees that the reparameterization $A$ is $C^\infty$-smooth on regular points. Indeed, assume up to a change of coordinates that $W\subset \mathbb{R}^{dim(M)}$ and that both $\varphi=(\varphi_t)_t$ and $\psi \in \mathcal{Z}^{\infty}(\varphi)$ are flows in $\mathbb{R}^{dim(M)}$. 
Consider $t\ne 0$ and, for $c\in \mathbb R$ and $z\in W$, set
$F_t(c, \, z) = \varphi_{c t} (z) - \psi_t ( z)$. Since $F_t$ is $C^\infty$-smooth, $F_t( A(z),  z) = 0$ for every $z\in W$ 
and $\dfrac{\partial F_t}{\partial c}(c,x) = t X(\varphi_{c t} (x)) \neq 0$ (because $x$ is a regular point), the implicit function theorem \cite{Pugh} 
assures that $A(\cdot)$ 
is a $C^\infty$-function in $W$.
\end{remark}

\subsection{Proof of Corollary~\ref{cor:centr}}

Let $\varphi$ be a $C^\infty$-expansive flow on a compact and connected Riemannian manifold $M$ whose singularities are hyperbolic and non-resonant and suppose without loss of generality that 
$\Lambda=\overline{\bigcup\limits_{\sigma \in \text{Sing}(\varphi)} W^s(\sigma) }$ (the other case is completely similar).
If $\psi \in \mathcal{Z}^{\infty}(\varphi)$, Theorem \ref{thm:flows} guarantees that there exists a continuous
$A: M \rightarrow \mathbb{R}$ such that $\psi_t(x) = \varphi_{A(x) t} ( x)$. Moreover, the arguments used in the proof of Lemma \ref{extencaosing} yield that this reparameterization $A$ is constant in each $\overline{W^s(\sigma_i)}$.
Since there are finitely many singularities the image of $A$ is constituted only by a finite number of elements.
Now, using that $M$ is connected and that $A$ is continuous we conclude that $A$ is constant. Consequently, there exist $c \in \mathbb{R}$ so that $\psi_t(x) = \varphi_{c t}(x)$, for every $t \in \mathbb{R}$ and $x \in M$. This proves the corollary.
\hfill $\square$

\section{Expansive $\mathbb R^d$-actions and suspensions}\label{sec:actions1}

In what follows we provide a characterization of expansive $\mathbb R^d$-actions obtained as suspensions
of $\mathbb Z^d$-actions.
Expansive subdynamics of $\mathbb Z^d$-actions on compact metric spaces have been
considered e.g. in \cite{BL,ELMW}. Here we deal with expansive $\mathbb{R}^d$-actions
and this first result is an extension of \cite[Theorem 6]{BW72}, where
Bowen and Walters proved that a continuous  $\mathbb{Z}$-action is expansive if and only if its suspension flow is 
C-expansive. We first recall the notion of suspension action.

\subsection{Suspension of $\mathbb Z^d$-actions}
We first recall the notion of an $\mathbb R^d$-action that is obtained as \emph{suspension of a $\mathbb Z^d$-action}.
Let $(e_i)_{i=1}^d$ be the canonical base on $\mathbb Z^d$. Given a $\mathbb{Z}^d$-action, $\varphi: \mathbb{Z}^d \times M \rightarrow M$ ($d \geq 2$) we construct $\mathbb R^d$-actions that are suspensions of $\varphi$.
Given a continuous roof function $R : M \rightarrow (0, \, \infty)^d$ with $R=(R_1, \dots, R_d)$ consider the set $M_R = (M \times \mathbb{R}_+^d) / \sim_R$ where $\sim_R$ is the equivalence relation 
\begin{center}
$(x, \, a_1, \, \cdots, \, a_{i-1}, \, R_i(x), \, a_{i+1}, \, \cdots, \,  a_d) \sim_F (f_i(x), \, a_1, \, \cdots, \, a_{i-1}, \, 0, \, a_{i+1}, \, \cdots, \,  a_d)$
\end{center}
for every $x\in M$, $0\le a_j \le R_j(x)$, where $f_i(x) = \varphi(e_i, \, x)$. Observe that 
$\varphi((n_1, \, \cdots, \, n_d), \, x) = f_1^{n_1} \circ \cdots \circ f_d^{n_d}(x)$ for every integers $n_i$ and $x\in M$.
The suspension $\mathbb R^d$-action of $\varphi: \mathbb{Z}^d \times M \rightarrow M$
with the roof function $R$ is the action $\Phi: \mathbb{R}^d \times M_R \rightarrow M_R$ defined by 
\begin{align*}
\Phi((t_1, \, \cdots, \, t_d), \, (x, \, a_1 , \, \cdots, \, a_d)) 
	& = \Phi(t_1 e_1 + \cdots +  t_d e_d, \, (x, \, a_1, \, \cdots, \, a_d)) \\
                        & = \Phi(t_1 e_1, \, \Phi(t_2 e_2, \, \cdots \Phi(t_d e_d, \, (x, \, a_1, \, \cdots, \, a_d)) ))
\end{align*}
where
$$\Phi(t \, e_i, \, (x, \, a_1, \dots, a_d)) = (f_i^n(x), \, a_1, \, \dots, \, a_{i-1}, \, a_i + t - \sum\limits_{j = 0}^{n-1} R_i(f_i^j(x)), \, a_{i+1}, \dots,  a_d)$$
and $n\in \mathbb Z$ is uniquely determined by $\sum_{j=0}^{n-1} R_i(f_i^j(x)) \le a_i + t < \sum_{j=0}^{n} R_i(f_i^j(x))$,
for every $x\in M$, $1\le i \le d$ and $0\le a_i \le R_i(x)$.
In this way, any $\mathbb Z^d$-action on a compact $n$-dimensional manifold $M$ determines a $\mathbb R^d$-action 
$\Phi$ on a $n+d$-dimensional compact manifold $M_{R}$.

The space $M_R$ is metrizable and we exhibit a pseudo-metric $d$ that is compatible with the natural topology on $M_R$
and is the analogous of  the Bowen-Walters metric for flows.
For the purposes of the upcoming Theorem~\ref{thm:suspensions}  it is enough to consider the roof function $R$ constant to one
and the corresponding space $M_1$. Let $\rho$ denote the metric on $M$.
Given $M \times \{(t_1, \, \cdots, \, t_d)\} \subset M_{{1}}$ 
and  $\sigma = (\sigma_1, \cdots, \, \sigma_d) \in \{0, \, 1\}^d$, consider the `horizontal' distance 
$\rho_h$ on $M \times \{(t_1, \, \cdots, \, t_d)\}$ defined by
\begin{align*}
& \rho_{h} ((x, \, t_1, \, \cdots, \, t_d), \, (y, \, t_1, \, \cdots, \, t_d)) \\
& \qquad= \!\!
\sum\limits_{\sigma \in \{0, \, 1\}^d}
\Big\{\prod\limits_{i = 1}^{d}\Big[\sigma_i \cdot t_i + (1 - \sigma_i) \cdot (1 - t_i) \Big]\Big\}
\rho(f_1^{\sigma_1} \circ \cdots \circ f_d^{\sigma_d}(x), \, f_1^{\sigma_1} \circ \cdots \circ f_d^{\sigma_d}(y)).
\end{align*}

It is not hard to show that
$$\sum\limits_{\sigma \in \{0, \, 1\}^d}\left\{\prod\limits_{i = 1}^{d}\Big[\sigma_i \cdot t_i + (1 - \sigma_i) \cdot (1 - t_i) \Big]\right\} = 1$$
and, consequently, defined in this way, $\rho_{h}((x_1, \, t_1, \, \cdots, \, t_d), \, (x_2, \, t_1, \, \cdots, \, t_d))$ consists
of the convex combination of the distances between the images of the points $x,y$ and their iterates 
by maps of the form $f_1^{\sigma_1} \circ \cdots \circ f_d^{\sigma_d}$. 
Second, in the particular case that $d = 1$, $\rho_{h}((x, \, t), \, (y, \, t)) = (1 - t)\rho(x, \, y) + t \rho(f(x), \, f(y))$ coincides with the metric introduced in \cite{BW72} for suspension flows.
Given any two points $(x, \, t_1, \, \cdots, \, t_d), \, (y, \, s_1, \, \cdots, \, s_d) \in M_1$ consider the space 
of all the finite (admissible) sequences 
$\omega_1=(x_1, \, t_1, \, \cdots, \, t_d),  \, \cdots, \, \omega_n = (x_n, \, s_1, \, \cdots, \, s_d)$ such that $x_1=x$, $x_n=y$ and, for each $1 \leq i \leq n-1$, either
\begin{enumerate}
\item $\omega_i, \, \omega_{i+1} \in M \times \{(t_1, \, \cdots, \, t_d)\}$ for some $(t_1, \, \cdots, \, t_d) \in [0, \, 1)^d$, in which case we set ${\tilde d}(\omega_i, \, \omega_{i+1}) = \rho_{h}(\omega_i, \, \omega_{i+1})$; or
\item $\omega_i, \, \omega_{i+1}$ belong to the same orbit by the action $\Phi$, and we define
${\tilde d}(\omega_i, \, \omega_{i+1}) :=\inf\{ \|v\| \colon \Phi_v ({\omega_i)= \omega_{i+1}} \}$ 
as the `vertical distance' between $\omega_i$ and $\omega_{i+1}$ in $M_1$.  

\end{enumerate}
Finally, consider the metric in $M_1$ given by 
$$
d((x, \, t_1, \, \cdots, \, t_d), \, (y, \, s_1, \, \cdots, \, s_d)) = \inf \sum\limits_{i = 1}^{n-1} {\tilde d}(\omega_i, \, \omega_{i+1}),
$$ 
where the infimum is taken over the space of previously defined admissible sequences
between $(x, \, t_1, \, \cdots, \, t_d)$ and $(y, \, s_1, \, \cdots, \, s_d)$. 

\subsection{Characterization of expansive $\mathbb R^d$-actions that are suspensions}
This section is devoted to the proof of the following characterization.

\begin{theorem} \label{thm:suspensions}
Let $M$ be a compact Riemannian manifold and $1 : M \rightarrow (0, \, \infty)^d$ be the roof function constant to one.
 A bi-Lipschitz  $\mathbb Z^d$-action $\varphi: \mathbb{Z}^d \times M \rightarrow M$ is expansive if and only if its suspension $\mathbb R^d$-action $\Phi: \mathbb R^d \times M_1 \to M_1$ is expansive.
\end{theorem}

\begin{proof}
Suppose that the action $(\Phi_v)_{v \in \mathbb{R}^d}$ is expansive (cf. Definition \ref{acaoexpansivacontinua}). So, given 
$0 < \varepsilon < \frac{1}{2}$ let $\delta > 0$ be so that if $x, \, y \in M$ satisfy $d(\Phi_v(x), \, \Phi_{h(v)}(y)) < \delta$ for every $v \in \mathbb{R}^d$ with respect to a continuous function $h: \mathbb{R}^d \rightarrow \mathbb{R}^d$ so that $h(0) = 0$,  then
$y = \Phi_{v_0}(x)$ for some $\|v_0\| < \varepsilon$. 
We show that the $\mathbb Z^d$-action $\varphi$ is expansive. Assume that $x_1, \, x_2 \in M$ are such that 
$\rho(\varphi_{(n_1, \, \cdots, \, n_d)}(x_1), \, \varphi_{(n_1, \, \cdots, \, n_d)}(x_2)) < \delta$ for all $(n_1, \, \cdots, \, n_d) \in \mathbb{Z}^d$.
If $[t]$ denotes the integer part of $t$ and $\{t_1\} = t - [t]$ is the fractional part of $t$, for every $t\in \mathbb R$, observe that
\begin{align*}
d(\Phi_{(t_1, \, \cdots, \, t_d)} & (x_1, \, 0, \, \cdots, \, 0), \, \Phi_{(t_1, \, \cdots, \, t_d)}(x_2, \, 0, \, \cdots, \, 0))  \\
&= d(\Phi_{([t_1], \, \cdots, \, [t_d])}(x_1, \, \{t_1\}, \, \cdots, \, \{t_d\}), \, \Phi_{([t_1], \, \cdots, \, [t_d])}(x_2, \, \{t_1\}, \, \cdots, \, \{t_d\})) \\
&\leq \rho_{h} (\Phi_{([t_1], \, \cdots, \, [t_d])}(x_1, \, \{t_1\}, \, \cdots, \, \{t_d\}), \\
& \qquad \qquad \qquad \qquad \Phi_{([t_1], \, \cdots, \, [t_d])}(x_2, \, \{t_1\}, \, \cdots, \, \{t_d\})) \\
&= \rho_{h}((f_1^{[t_1]} \circ \cdots \circ f_d^{[t_d]}(x_1), \, \{t_1\}, \, \cdots, \, \{t_d\}), \\
& \qquad \qquad \qquad \qquad  (f_1^{[t_1]} \circ \cdots \circ f_d^{[t_d]}(x_2), \, \{t_1\}, \, \cdots, \, \{t_d\})) \\
&=  \sum\limits_{\sigma \in \{0, \, 1\}^d}\left(\prod\limits_{i = 1}^{d} \sigma_i \cdot \{t_i\} + (1 - \sigma_i) \cdot (1 - \{t_i\}) \right) \\
& \qquad \qquad \qquad \qquad \cdot \rho(f_1^{[t_1]+\sigma_1} \circ \cdots \circ f_d^{[t_d]+\sigma_d}(x_1), \, f_1^{[t_1]+\sigma_1} \circ \cdots \circ f_d^{[t_d]+\sigma_d}(x_2))
\\
& < \sum\limits_{\sigma \in \{0, \, 1\}^d}\left(\prod\limits_{i = 1}^{d} \sigma_i \cdot \{t_i\} + (1 - \sigma_i) \cdot (1 - \{t_i\}) \right) \cdot \delta = \delta
\end{align*}
for every $(t_1, \dots, t_d) \in \mathbb R^d$.
The expansiveness condition assures that $(x_2, \, 0, \, \cdots, \, 0) = \Phi_{v_0}(x_1, \, 0, \, \cdots, \, 0)$ for some $v_0 \in \mathbb{R}^d$ such that $\|v_0\| < \varepsilon < 1/2$. This implies that $x_1 = x_2$ and so the action $\varphi: \mathbb{Z}^d \times M \rightarrow M$ is expansive.

\medskip

Conversely, suppose that $\varphi$ is expansive. 
In particular $\varphi$ is also expansive with respect to the distance
$${\tilde \rho}(x_1, \, x_2) = \min\limits_{\sigma \in \{0, \, 1\}^d}\{\rho(f_1^{\sigma_1} \circ \cdots \circ f_d^{\sigma_d}(x_1), \, f_1^{\sigma_1} \circ \cdots \circ f_d^{\sigma_d}(x_2))\}$$ and let $\zeta>0$ be such a constant of expansiveness. Indeed,
since $\varphi$ is bi-Lipschitz then there exists $C>0$ so that $\frac{1}{C}\tilde \rho(x , y) \leq \rho(x , y) \leq C \tilde \rho(x , y)$ for all $x,y\in M$.
Now, given $\varepsilon > 0$ take $0<\delta < \min\{\varepsilon,\frac14, \zeta\}$. 
Suppose that $d(\Phi_v(x_1, \, t_1, \, \cdots, \, t_d), \, \Phi_{h(v)}(x_2, \, s_1, \, \cdots, \, s_d)) < \delta$ for all $v \in \mathbb{R}^d$ and for some continuous map $h: \mathbb{R}^d \rightarrow \mathbb{R}^d$ such that $h(0) = 0$.
We may assume without loss of generality that $y_1 = (x_1, 1/2, \, \cdots, \, 1/2)$ and $y_2 = (x_2, \, s_1, \, \cdots, \, s_d)$ in the coordinates of $M \times [0, \, 1]^d$ (if $y_1$ is not is in the form $(x_1, 1/2, \, \cdots, \, 1/2)$ just take 
$\|w\| \leq 1/2$ such that $\Phi_w(y_1) = (x_1, 1/2, \, \cdots, \, 1/2)$ and consider the points $\Phi_w(y_1)$ and 
$\Phi_{h(w)}(y_2)$).
Observe that 
$$
{\tilde \rho}(x_1, \, x_2) \leq d(y_1, \, y_2) = d(\Phi_0(y_1), \, \Phi_{h(0)}(y_2)) < \delta < 1/4.
$$
Now, suppose that $d(\Phi_v(y_1), \, \Phi_{h(v)}(y_2)) < \delta < 1/4$  for all $v \in \mathbb{R}^d$. In particular, taking $v = e_i$ ($1 \leq i \leq d$) 
it holds that
${\tilde \rho}(f_i^n(x_1), \, f_i^n(x_2)) \leq d(\Phi_{n \cdot e_i}(y_1), \, \Phi_{h(n \cdot e_i)}(y_2)) < \delta < 1/4$
for every $n \in \mathbb{Z}$.
Proceeding recursively, we obtain that 
\begin{align*}
{\tilde \rho}((f_1^{n_1} \circ \cdots \circ f_d^{n_d})(x_1), & \, (f_1^{n_1} \circ \cdots \circ f_d^{n_d})(x_2)) \\
	& \leq d(\Phi_{(n_1, \, \cdots, \, n_d)}({y_1}), \, \Phi_{h(n_1, \, \cdots, \, n_d)}({y_2})) < \delta
\end{align*} 
for all $(n_1, \, \cdots, \, n_d) \in \mathbb{Z}^d$.
Finally, by the expansiveness of $\varphi$ we obtain that $x_1 = x_2$, implying $y_2 = \Phi_v(y_1)$ for some $v \in \mathbb{R}^d$ such that $\|v\| < \delta < \varepsilon$.

\end{proof}

\section{Centralizer for expansive homogenenous $\mathbb R^d$-actions}\label{sec:homogeneous}\label{sec:actions2}

In this section we prove that  $\mathbb R^d$-actions with expansive elements have one dimensional orbits (Corollary~\ref{cor:actions}) and study the centralizer of homogeneous expansive $\mathbb R^d$-actions (Theorem~\ref{centacaoespansiva}). In what follows, $\|\, \cdot \,\|$ will denote the Euclidean norm in $\mathbb{R}^d$.

\subsection{Proof of Corollary~\ref{cor:actions}}

Let $M$ be a compact Riemannian manifold and let $\Phi: \mathbb{R}^d \times M \rightarrow M$ be a continuous action in $M$ such that the flow $(\Phi_{t \, v})_{t \in \mathbb{R}}$ is Komuro-expansive for a fixed $v \in \mathbb{R}^d$. Consider $\{v, \, u_2, \, u_3, \, \cdots, \, u_d \}$ a basis of $\mathbb{R}^d$ containing the vector $v$.

Observe that $\Phi_{t \, v} \circ \Phi_{s \, u_i} = \Phi_{s \, u_i} \circ \Phi_{t \, v}$ for all $t, \, s \in \mathbb{R}$ and $2 \leq i \leq d$. So, by expansiveness of the flow $(\Phi_{t \, v})_{t \in \mathbb{R}}$, as a consequence of Theorem \ref{thm:flows} for each $2 \leq i \leq d$ there exists a unique function $A_i : X \rightarrow \mathbb{R}$ invariant along orbits of flow $(\Phi_{t \, v})_{t \in \mathbb{R}}$ and such that $\Phi_{u_i}(t, \, x) = \Phi_{v}(A_i(x)t, \, x)$. Consequently,
$$\Phi(t_1 v + t_2 u_2 + \cdots + t_d u_d, x) = \Phi_{v}((1 + A_2(x) + \cdots + A_d(x))t , \, x),$$
which proves that all regular orbits of $\Phi$ are unidimensional.
This completes the proof of the corollary.

\subsection{Proof of Theorem~\ref{centacaoespansiva}}

In this subsection we characterize the space of $C^1$ $\mathbb{R}^d$-actions $\Psi$ that commute with 
an expansive  $C^1$ $\mathbb{R}^d$-action $\Phi$. Our purpose is to prove that the action $\Psi$ is a reparameterization of $\Phi$: there exists a continuous map $A: \mathbb T^n \rightarrow \mathcal M_{d\times d}(\mathbb{R})$ satisfying: (i) $A(x) = A(\Phi_v(x))$ for every $v\in\mathbb R^d$ and $x\in \mathbb T^n$, and (ii) $\Psi_v(x) = \Phi(A(x) v, \, x)$  for every $(v, \, x) \in \mathbb{R}^d \times \mathbb T^n$ 
(cf. Proposition~\ref{Acomblinear} below).
Since the strategy of the proof is similar to the one of Theorem~\ref{thm:flows} we will sketch the details and highlight
the main differences. The starting point is the following canonical form for commuting vector fields, similar to the tubular neighborhood 
theorem.

\begin{lemma}[Lee \cite{Lee}, Theorem 18.6]\label{acaotubular} Let $M$ be a smooth $n$-manifold, let $d<n$ and let $\Phi$ be a 
$C^1$ $\mathbb{R}^d$-action on $M$. Assume that $\Phi$ is generated by smooth commuting linearly independent vector fields $X_1, \, \cdots, \, X_d$ on some open subset $W \subseteq M$. For each $p \in W$ there exists an open neighborhood $U$ of $p$, 
a $C^1$-diffeomorphism $h~:~U~\rightarrow h(U) \subset \mathbb{R}^n$  with coordinate functions $h(q) = (s_1(q), \, \cdots, \, s_n(q))$ on $U$ and $h(p) = 0$ and such that $X_i = h^{-1}_* \frac{\partial}{\partial s_i}$ for $i = 1, \cdots, \, d$. 
 If $S \subseteq U$ is an embedded codimension-$d$ submanifold and $q$ is a point of $S$ such that $T_qS$ is complementary to  
 $\text{span}(X_1(q), \, \cdots, \, X_d(q))$ then the coordinates can be chosen such that $S$ is  defined by the coordinates $s_1 = \cdots = s_d = 0$.
\end{lemma}

The next lemma asserts that the leaves formed by the orbits of expansive $\mathbb R^d$-actions do not admit closed curves of arbitrarily small diameter.  Since the proof of the lemma is completely similar to the one of Lemma~\ref{e0}, making use of Lemma~\ref{acaotubular}, we shall omit it.

\begin{lemma}\label{e0acao} If $\Phi: \mathbb{R}^d \times \mathbb T^n \rightarrow \mathbb T^n$ is an $C^1$ expansive homogeneous $\mathbb{R}^d$-action, then
\begin{center}
$\varepsilon_0(\Phi) = \inf\{\|v\| > 0: \; v$ is period of $\Phi\} > 0$.
\end{center}
\end{lemma}

We should observe that if an $\mathbb R^d$-action is not homogeneous then $\varepsilon_0(\Phi)$ could be zero
and the (local) geometry of the space of orbits in the case of non-homogeneous actions can be very complicated.

\begin{lemma}\label{singepsilon} If $\Phi: \mathbb{R}^d \times \mathbb T^n \rightarrow \mathbb T^n$ is an expansive $C^1$-action and $\Psi \in \mathcal{Z}^{1}(\Phi)$ then, for all $0 < \varepsilon < \varepsilon_0(\Phi)/3$, there exists $\mu > 0$ and a unique map $z :\, \overline{B_{\mu}(0)} \times \mathbb T^n \rightarrow B_{\varepsilon}(0) \subset \mathbb{R}^d$ such that $\Psi_s(x) = \Phi(z(s, \, x), \,x)$ for all $(s, \, x) \in \overline{B_{\mu}(0)} \times \mathbb T^n$.
Moreover,
\begin{itemize}
\item[(I)\; ] $z$ is a continuous map,
\item[(II) ] If $v$, $u$, $u + v \in \overline{B_{\mu}(0)}$, then $z(u + v, \, x) = z(u, \, x) + z(v, \, \Psi(u, \, x))$.
\end{itemize}
\end{lemma}

\begin{proof} 
Although the proof is analogous to the one of Lemma~\ref{singepsilon1}, we include the construction of
the local reparameterization for completeness.
Let $\varepsilon_0 > 0$ be as in Lemma \ref{e0acao}, take $0 < \varepsilon < \varepsilon_0(\Phi)/3$ and let
$\delta>0$ be given by expansiveness (recall Definition \ref{acaoexpansivacontinua}).
By compactness of $\mathbb T^n$ there exists $\mu > 0$ such that if $0 < \varepsilon < \varepsilon_0/3$, then 
$$
\sup\limits_{\|u\| \leq \mu}\left\{d(Id, \, \Psi_u)\right\} < \delta.
$$ 
If $\|u\| \leq \mu$ then 
$
d(\Phi_v(x), \, \Phi_v(\Psi_u(x)))  
                                          = d(\Psi_0(\Phi_v(x)), \, \Psi_u(\Phi_v(x))) 
                                          < \delta
$
for all $(v, \,x) \in \mathbb{R}^d \times \mathbb T^n$. Since $\Phi$ is an expansive action and $d(\Phi_v(x), \, \Phi_{h(v)}(\Psi_u(x))) < \delta$ for all $v \in \mathbb{R}^d$ (with $h = Id$), there exists $v_0 \in \mathbb{R}^d$ such that $\Phi_{v_0}(\Psi_u(x)) = \Phi_{v_0 + \eta}(x)$ for some $\eta \in B_{\varepsilon}(0)$. This implies that $\Psi_u(x) = \Phi_{\eta}(x)$ for some
vector $\eta$ satisfying $\|\eta\| < \varepsilon$. In particular, $\Psi_u(x)$ belongs to the orbit of $x$  relative to the $\mathbb R^d$-action $\Phi$.

This defines a map $z :\, \overline{B_{\mu}(0)} \times \mathbb T^n \rightarrow B_{\varepsilon}(0)$ such that $\Psi(u, \, x) = \Phi(z(u, \, x), \, x)$ for any  $(u, \, x) \in \overline{B_{\mu}(0)} \times \mathbb T^n$. To prove the uniqueness, observe that if $z_1, \, z_2 :\, \overline{B_{\mu}(0)} \times \mathbb T^n \rightarrow B_{\varepsilon}(0)$ are such that $\Phi(z_1(u, \, x), \, x) = \Psi(u, \, x) = \Phi(z_2(u, \, x), \, x)$, then $\Phi(z_1(u, \, x) - z_2(u, \, x), \, x) = x$ where $\|z_1(u, \, x) - z_2(u, \, x)\| \leq \|z_1(u, \, x)\| + \|z_2(u, \, x)\| < 2 \, \varepsilon_0(\Phi)/3$. Since this contradicts the non existence of periods smaller than $\varepsilon_0(\Psi)$ the uniqueness of $z$ follows. 
The proof of the continuity is completely analogous to the one of Lemma~\ref{singepsilon1} and we shall omit it.
\end{proof}


Next, we will construct an extension to $\mathbb{R}^d \times \mathbb T^n$ for the continuous reparameterization  described in Lemma \ref{singepsilon}. More precisely  we have the following:

\begin{lemma}\label{lemaextensaoacao} If $\Phi: \mathbb{R}^d \times \mathbb T^n \rightarrow \mathbb T^n$ is a continuous action and $\Psi : \mathbb{R}^d \times \mathbb T^n \rightarrow \mathbb T^n$ is a continuous action such that for $\mu > 0$ fixed there exists a reparameterization $z :\, \overline{B_{\mu}(0)} \times \mathbb T^n \rightarrow B_{\varepsilon}(0)$ such that $\Psi(v, \, x) = \Phi(z(v, \, x), \, x)$ for any $(v, \, x) \in \overline{B_{\mu}(0)} \times \mathbb T^n$, when $0 < \varepsilon < \varepsilon_0(\Phi)/3$, then exists a unique continuous function $p :\, \mathbb{R}^d \times \mathbb T^n \rightarrow \mathbb{R}^d$ where is extension of $z$ and such that $\Psi(s, \, x) = \Phi(p(s, \, x), \, x)$ for all $(s, \, x) \in \mathbb{R}^d \times \mathbb T^n$.
\end{lemma}

\begin{proof} 
The strategy of the proof is similar to the one of Proposition~\ref{extencao}, that is, to extend the local
reparameterization given in Lemma \ref{singepsilon} to a function $p : \mathbb{R}^d \times \mathbb T^n \rightarrow \mathbb{R}^d$ such that $\Psi(s, \, x) = \Phi(p(s, \, x), \, x)$ for all $(s, \, x) \in \mathbb{R}^d \times \mathbb T^n$. 
It is worthwhile to note that although there are are several ways to extend the reparameterization it is unique.
 The extension here is made radial by considering observing vectors 
in $\mathbb{R}^{d}$ as multiples of vectors in the unit sphere  $\mathbb{S}^{d-1}$. We shall
sketch the main differences and omit some details. 
For this, let $\mu>0$ be given by Lemma~\ref{singepsilon}, let $N \in \mathbb{N}$ such that $2^{-N} < \mu$ and fix
$v \in \mathbb{S}^{d-1}$.
For each $k \in \mathbb{N}$ let $D_k = \{z \in \mathbb{R}^d: \, \frac{k}{2^ N} \leq \|z\| \leq \frac{k+1}{2^N}\}$, which
contain the vectors of the form $u = t v \in \mathbb{R}^d$ for $t \in [k/2^N, \, (k+1)/2^N]$ 
Now, consider the functions $z_k: \, D_k \times \mathbb T^n \rightarrow \mathbb{R}^d$ given by
\begin{equation}\label{defzk}
z_k(t \cdot v, \, x) = z((t - k/2^N) \cdot v, \, x) + \sum\limits_{i=1}^k z(1/2^N \cdot v, \Psi((t - i/2^N) \cdot v, \, x)),
\end{equation}
\noindent{for every $k \in \mathbb{N}$.}
By Lemma \ref{singepsilon} and the definition of $z_k$, it follows that $z_k$ is continuous and satisfies 
\begin{align}\label{defzk1}
 z_k\big(\frac{k+1}{2^N} \cdot v, \, x\big)  = z_{k+1}\big(\frac{k+1}{2^N} \cdot v, \, x\big)
 	\quad\text{and}\quad 
 \Psi(t v, \, x) = \Phi(z_k(t v, \, x), \, x)
\end{align}
for all $x \in \mathbb T^n$, $k \in \mathbb{N}$ and $t \in [k/2^N, \, (k+1)/2^N]$. 
This allows to define the continuous map $p : \mathbb{R}^d \times \mathbb T^n \rightarrow \mathbb{R}^d$ given by
$$p(t \cdot v, \, x) = \left\{ \begin{array}{cl}
z(t \cdot v, \, x),              & \mbox{if } t \in [0, \, 1/2^N], v \in \mathbb{S}^{d-1} \\
z_k(t \cdot v, \, x),            & \mbox{if } t \in \left[k/2^N, \, (k+1)/2^N\right], \, v \in \mathbb{S}^{d-1}, \, k \in \mathbb{N}\\
\end{array} \right.,$$
The continuity of $p$  follows from relation (\ref{defzk1}) and Lemma~\ref{singepsilon}. Moreover, since 
$v\in \mathbb{S}^{d-1}$ was chosen arbitrary then Lemmas~\ref{singepsilon} and \ref{lemaextensaoacao} imply 
that $p$ satisfies
$\Psi(t  \cdot v, \, x) = \Phi(p(t \cdot v, \, x), \, x)$ for all $x \in \mathbb T^n$, $t \in \mathbb{R}$, $v \in \mathbb{S}^{d-1}$.

To prove the uniqueness of the reparameterization $p$, assume that there are continuous reparameterizations 
$p_1, \, p_2 :\, \mathbb{R}^d \times \mathbb T^n \rightarrow \mathbb{R}^d$ that extend $z$ and such that 
$\Phi(p_1(u, \, x), \, x) = \Psi(u, \, x) = \Phi(p_2(u, \, x), \, x)$ for any $(u, \, x) \in \mathbb{R}^d \times \mathbb T^n$.
Fix $x \in \mathbb T^n$ and let $\alpha_x(u) = p_1(u, \, x) - p_2(u, \, x)$. Observe that $\alpha_x^{-1}(0) \supset \overline{B_{\mu}(0)}$ and, consequently, $\alpha_x^{-1}(0) \neq \emptyset$. Moreover, since $\alpha_x$ is continuous then $\alpha_x^{-1}(0)$ is a closed subset of $\mathbb R^d$.
We claim that $\alpha_x^{-1}(0)=\mathbb{R}^d$.  If $\alpha_x^{-1}(0) \neq \mathbb{R}^d$, there would be a unit vector $v \in \mathbb{S}^{d-1}$ and $t_0 = \sup\{t > 0: \, t \cdot v \in \alpha_x^{-1}(0)\} <\infty$. Since $\alpha_x^{-1}(0)$ is a closed subset then $ t_0 \cdot v \in \alpha_x^{-1}(0)$.
Recalling the previous discussion, setting $x{'} = \Phi(p_1(t_0 \cdot v, \, x), \, x) \, (= \Phi(p_2(t_0 \cdot v, \, x), \, x))$, 
it follows that $\alpha_{x^{'}}$ is identically zero in $\overline{B_{\mu}(0)}$. Thus $\alpha_x$ is identically zero in 
$\{t \cdot v: \, t \in [0, \, t_0 + \mu]\}$, which contradicts the maximality of $t_0$.
This proves the claim and the uniqueness of the reparameterization. 
\end{proof}

The next lemma will complete the proof of Theorem~\ref{centacaoespansiva}.

\begin{lemma}\label{lemainvorbitas} 
Let $p$ be the reparameterization given by Lemma \ref{lemaextensaoacao}. Then $p$ is invariant along the orbits of $\Phi$, that is, $p(v, \, x) = p(v, \, \Phi(u, \, x))$ for all $u, \, v \in \mathbb{R}^d$ and $x \in \mathbb T^n$.
Moreover, there exists a $C^1$-map $\mathbb T^n \ni x \mapsto A(x) \in \mathcal M_{d\times d}(\mathbb{R})$ 
so that $p(v, x)=A(x) v$ for all $x\in \mathbb T^n$ and $v\in \mathbb R^d$.
\end{lemma}

\begin{proof} 
The arguments of Lemma~\ref{lemainvorbitas1} yield
$\Phi(p(v  + u, \, x), \, x)  = \Phi(p(v, \, x) + p(u, \, x), \, x)$
for every $u,v \in \mathbb{R}^d$ and $x\in \mathbb T^n$.
The uniqueness of $p$ implies that 
$p(v  + u, \, x) = p(v, \, x) + p(u, \, x) \quad \text{for all $u, \, v \in \mathbb{R}^d$ and $x\in \mathbb T^n$}$.
By the later and the
continuity of $p$ we conclude that there exists a continuous map $\mathbb T^n \ni x \mapsto A(x) \in \mathcal M_{d\times d}(\mathbb{R})$ 
so that $p(v, x)=A(x) v$ for all $x\in \mathbb T^n$ and $v\in \mathbb R^d$. The differentiability of $A$ follows from the
implicit function theorem as in the end of the proof of Theorem~\ref{thm:flows}. This finishes the proof of the lemma.
\end{proof}

In the remaining of this section we provide a geometric characterization of the linear reparameterization obtained 
in Theorem~\ref{centacaoespansiva}, which is of independent interest in the case of a
$C^1$ and expansive $\mathbb R^d$-action on a compact manifold $M$ of dimension $n$ larger than $d$. We
prove that the reparameterization can be written as a matrix of change of coordinates.

\begin{proposition}\label{Acomblinear} 
Let $\Phi: \mathbb{R}^d \times M \rightarrow \mathbb T^n$ be a $C^1$-expansive and homogeneous action, $n\ge d$, and 
let $\Psi \in \mathcal{Z}^{1}(\Phi)$ be such that $\Psi(v, \, x) = \Phi(A(x) v, \, x)$ for all $v\in \mathbb R^d$ and $x\in \mathbb T^n$.
Considering the vector fields $X_i(\cdot) = \frac{d \Phi(t e_i,\cdot)}{dt} \mid_{t=0}$ and $Y_i(\cdot) = \frac{d \Psi(t e_i,\cdot)}{dt} \mid_{t=0}$ for every $1\le i \le d$, for each $x \in \mathbb T^n$ the linear map $A(x)$ is represented by the matrix of representation of the vectors  $(Y_i(x))_{1\le i \le d}$ on the basis $(X_i(x))_{1\le i \le d}$. 
\end{proposition}

\begin{proof}
The homogeneity assumption assures that the vector fields $X_1, \, \cdots, \, X_d$ are linearly independent.
Fix $x\in \mathbb T^n$. Up to a change of coordinates we may assume without loss of generality
that $\Phi$ is (locally) an $\mathbb R^d$-action on an open set of $\mathbb R^n$ and that 
 $X_i(z)=e_i$ for all $1\le i \le d \; (< n)$. Indeed, Lemma \ref{acaotubular} guarantees that there exists 
an open neighborhood $V_x \subset \mathbb T^n$ of $x$, and a change of coordinates 
$h: V_x \rightarrow h(V_x)  \subset \mathbb{R}^n$ so that $X_i = h^{-1}_* e_i$ for all $1\leq i \leq d$. 
and, still denoting by $\Phi$ the induced action on $h_x(V)$,  one can write
$$
\Phi(v, z)=z + \sum\limits_{i=1}^d v_i \cdot X_i 
$$
for every $z = h(x)\in \mathbb R^d$ and every small values $(v_i)_{1\le i \le d}$ such that 
$\Phi(v, z)$ belongs to $h(V_x)$. Reducing $V_x$ if necessary,  let $\delta>0$ be such that 
$h(V_x)=[-\delta,\delta]^n$.

The first step in the proof of Theorem \ref{centacaoespansiva} implies that the orbits of $\Phi$ are fixed by any 
element $\Psi \in \mathcal{Z}^{1}(\Phi)$. Moreover, since $\Phi$ restricted to each of its orbits is generated by the $d$ 
constant vector fields $X_i(z) = e_i$ ($1 \leq i \leq d$) then $\Phi$ acts in each of its (local) orbits as a group of translations 
in $\mathbb{R}^d$. Indeed, in the linearization coordinates, the (local) orbit of $z\in h(V_x)$ is 
$\cF(z)=\{ w\in [-\delta,\delta]^n \colon w = z + t e_i \, \text{for some} \, 1\le i \le d \; \text{and}\, t\in \mathbb R\}$.
Then, for any given  small translation vector $g$ in $\cF(z) \simeq \mathbb{R}^d$ there exists a vector $u = u_g$ such that $\Phi(u_g, \, z) = z+ g$. 
In summary, for every $z \in [-\delta,\delta]^n$ the map $\Psi\mid_{\cF(z)}$ commutes with all local translations in $\cF(z)$.

Now we prove that in these linearization coordinates the action $\Psi \in \mathcal{Z}^1(\Phi)$ is also a translation along the orbit $\cF(z)$. Since $\cF(z)\simeq [-\delta,\delta]^n \subset \mathbb R^n$ this is an immediate consequence of the following:

\medskip \noindent 
{\bf Claim:} \emph{
If $f: \mathbb{R}^d \rightarrow \mathbb{R}^d$ is $C^1$
and commutes with all translations in $\mathbb{R}^d$ then $f$ is itself a translation in $\mathbb{R}^d$.}

\begin{proof}[Proof of the claim]
One can write in coordinate functions  
$$f(x_1, \, x_2, \, \cdots, \, x_d) = (f_1(x_1, \, x_2, \, \cdots, \, x_d), \, \cdots, \, f_d(x_1, \, x_2, \, \cdots, \, x_d) )$$
for $(x_1, \, x_2, \, \cdots, \, x_d) \in \mathbb{R}^d$.  Since $f$ commutes with all the translations of form $x + \lambda \, e_j$, for $x = (x_1, \, \cdots, \, x_d)$, $\lambda \in \mathbb{R}$ and where $\{e_1, \, \cdots, \, e_d\}$ denotes the
canonical basis in $\mathbb{R}^d$, then $f(x + \lambda \, e_j) = f(x) + \lambda \, e_j$. Analyzing each coordinate
independently, this means that $f_i(x + \lambda \, e_j) = f_i(x) + \lambda \, \delta_{ij}$ for all $1\le i, j \le d$, where 
$\delta_{ij}=1$ if $i=j$ and $\delta_{ij}=0$ otherwise. Thus, 
$$
\dfrac{\partial f_i}{\partial x_j} (x) = \lim\limits_{\lambda \rightarrow 0}\frac{f_i(x + \lambda \, e_j) - f_i(x)}{\lambda} = \delta_{ij}
$$ 
and, consequently, $f_i(x) = x_i + r_i$ for $r_i=f_i(0) \in \mathbb R$.  This guarantees that $f(x_1, \, x_2, \, \cdots, \, x_d) = (x_1, \, x_2, \, \cdots, \, x_d) + (r_1, \, r_2, \, \cdots, \, r_d)$ is a translation and completes the proof of the claim.

\end{proof}

We are now in a position to complete the proof of the proposition. 
The previous claim implies that $\Psi(u, \, y)$ is a translation for all $(u, \, y) \in U$. In particular, if $\{u_1, \, \cdots, u_d\}$ is a basis of $\mathbb{R}^d$, the flows $(\Psi_{t \cdot u_i})_{t \in \mathbb{R}}$ are flows of translations. Indeed, by the 
group property of $(\Psi_{t \cdot u_i})_{t \in \mathbb{R}}$  there exists a vector $w_i$ such that $\Psi_{t \cdot u_i}(x) = x + t w_i$ and consequently the $d$ vector fields  $Y_1, \, Y_2, \, \cdots, \, Y_d$ defining $\Psi$ are constant.
Then, one can write
$$\left\{
\begin{array}{c}
Y_1 = a_{11} X_1 + a_{21} X_2 + \cdots + a_{d1} X_d \\
Y_2 = a_{12} X_1 + a_{22} X_2 + \cdots + a_{d2} X_d \\
                       \vdots                       \\
Y_d = a_{1d} X_1 + a_{2d} X_2 + \cdots + a_{dd} X_d \\
\end{array} \right.
$$
on $U$ and let
$$
A = \left[
\begin{array}{cccc}
a_{11} & a_{12} & \cdots & a_{1d} \\
a_{21} & a_{22} & \cdots & a_{2d} \\
\vdots & \vdots & \ddots & \vdots \\
a_{d1} & a_{d2} & \cdots & a_{dd} \\
\end{array} \right].
$$
Then, writting $v= \sum_{i=1}^d r_i Y_i \in \mathbb R^d$ as a linear combination of the vectors on the
base $(Y_i)_{1\le i \le d}$,
\begin{align*}
\Psi(v, \, z) & 
		= z + \sum\limits_{j = 1}^d r_j Y_j
               = z + \sum\limits_{j = 1}^d r_j \Big[\sum\limits_{i = 1}^d a_{ij} X_i \Big]  \\
               & = z + \sum\limits_{i = 1}^d \Big[ \sum\limits_{j = 1}^d r_j \, a_{ij} \Big] \cdot X_i 
               = z + A  v 
               = \Phi(A v, \, z)
\end{align*}
for all $z\in U$. This proves that $A(x) = Dh(x)^{-1} \, A \, Dh(x)$ where $A$ is the previous matrix of change of coordinates
that determines the action $\Psi$ as reparameterization of action $\Phi$. This finishes the proof of the proposition.
\end{proof}

\section{Examples and applications}\label{sec:examples}

This section is devoted to present some applications of our main results and a discussion on other notions of
expansiveness. 

\begin{example}(Suspension flows)
Given a homeomorphism $f: M \to M$ on a compact metric space $M$ and a continuous roof function $r: M \to \mathbb R^+$ that is bounded away from zero consider the quotient space 
$$
M_r=\{(x,s) \in M \times \R^{+} : 0 \leq s \leq r(x)\} / \sim
$$
obtained by the equivalence relation that $(x,r(x)) \sim (f(x),0)$ for every $x \in M$. The
suspension flow $(\varphi_t)_t$ on $M_r$ associated to $(f, M, r)$ is defined by 
the ``vertical displacement'' $\varphi_t(x,s)=(x,t+s)$ whenever the expression is well defined. More precisely,
$
\varphi_t(x,s)
	=\big(f^k(x), t+s-\sum_{j=0}^{k-1} r(f^j(x)) \big)
$
where $k=k(x,t,s)\in \mathbb Z$ is determined by
$
\sum_{j=0}^{k-1} r(f^j(x)) \leq t+s < \sum_{j=0}^{k} r(f^j(x)).
$
Clearly $M\times\{0\}$ is a global cross-section to the flow. It follows from \cite{BW72} that $f$ is expansive if and only if the flow $(\varphi_t)_t$ is $C$-expansive. In particular, from Theorem~\ref{thm:flows}, the suspension flow of any expansive homeomorphism (e.g. quasi-Anosov diffeomorphisms with intermittency) or Axiom A  flows restricted to the non-wandering set have quasi-trivial centralizers.
\end{example}

\begin{example}(Lorenz attractors) \label{ex:Lorenz}
The Lorenz equations correspond to the system of polynomial ordinary differential equations in $\mathbb{R}^3$ 
\begin{equation}\label{lorenz}
\left\{ \begin{array}{l}
\frac{dx}{dt} = a (y - x) \\
\frac{dy}{dt} = - x z + r x - y \\
\frac{dz}{dt} = x y - b z\\
\end{array}\right.,
\end{equation}
with parameters $a, \, b, \, r \in \mathbb{R}$. Computer simulations led Lorenz \cite{Lorenz} to propose the existence 
of a ``strange attractor''  for the parameters $a = 10$, $b = 8/3$ and $r = 28$. 
For the classical parameters proposed by Lorenz, the three singularities of the equation \eqref{lorenz} are hyperbolic, and $\sigma_0$ belongs to the ``chaotic attractor'' and is accumulated by orbits of regular points.  
Simple computations yield that 
the eigenvalues of $\sigma_0$ are $\dfrac{-11 -\sqrt{1201}}{2} \approx -22,83$; $-\frac{8}{3} \approx -2,67$ and $\dfrac{-11 +\sqrt{1201}}{2} \approx 11,83$.
By the symmetry of the equations \eqref{lorenz}, the eigenvalues of $\sigma_1$ and $\sigma_2$ are the same.  
The singularity $\sigma_1$ 
has a real eigenvalue $\lambda \approx -13,85$ and two complex conjugates eigenvalues $z, \bar z$
where $z\approx 0.09 + 10.19 i$. 
In particular the singularities of \eqref{lorenz} satisfy the non-resonant conditions. Indeed, 
the singularity $\sigma_0$ is non-resonant since the unstable subspace is one-dimensional and the stable subspace 
of $\sigma_0$ is non-resonant because one eigenvalue is rational and the other is irrational. Finally,
the singularities $\sigma_1$ and $\sigma_2$ are non-resonant since their stable subspace 
is one-dimensional and has a pair of complex conjugate eigenvalues of along the unstable subspace.

In order to be able to describe the dynamical features of the `chaotic attractor' associated to the ODE~\eqref{lorenz},
geometric Lorenz attractors were introduced independently in \cite{lorAfraimovich,lorGuckenheimer} (see
Figure~3 below).
These form a parametrized family of vector fields, whose parameters correspond to the real eigenvalues $\lambda_1 < \lambda_2 <0 <-\lambda_2 <\lambda_3$ at the singularity $\sigma_0=(0,0,0)$. 

There exists a $C^1$-open subset of vector fields 
$\cU \subset \mathfrak{X}^\infty(\mathbb R^3)$ and an open elipsoide $V \subset \mathbb R^3$ containing the origin 
such that every $X\in \cU$ exhibits  a  geometric Lorenz attractor $\Lambda_X = \bigcap_{t\ge 0} \overline{X_t(V)}$, which is 
a partially hyperbolic attractor and whose restriction of the flow to the attractor
is Komuro expansive (see e.g. \cite{vitorzeze} for precise definitions and proofs). 
Such construction can be performed in an open domain of a compact manifold $M$ and if this is the case we 
will say that $X\in \mathfrak{X}^\infty(M)$ has a geometric Lorenz attractor. 
Since the non-ressonance condition for the singularity $\sigma_0$ is satisfied for both the original parameters proposed
by Lorenz and is a $C^1$-open and $C^\infty$ dense condition on the space of vector fields in $\cU$, the following is an immediate consequence of Corollary \ref{cor:centr}: 

\begin{corollary}
Let $\mathcal U \subset \mathfrak{X}^\infty(M)$ be an open set of vector fields so that every $X\in U$ has a geometric
Lorenz attractor $\Lambda_X$.
Then there exists a $C^1$-open and $C^\infty$-dense subset $\cU'\subset \cU$ so that, every vector field
$X\in \cU'$ admits a geometric Lorenz attractor $\Lambda_X$ whose centralizer on its topological 
basin of attraction is trivial.
\end{corollary}
\end{example}

We observe that the argument used in the previous example extends to a more general class
of three-dimensional flows. 

\begin{example}(Robustly transitive three-dimensional sets)
In \cite{MPP}, the authors 
described the structure of all $C^1$ robustly transitive sets with singularities 
for flows on compact Riemannian three-dimensional manifolds. These are partially hyperbolic attractors (or repellers) 
for the vector field with volume-expanding central direction and have an invariant foliation whose 
leaves are forward contracted by the flow, and has positive Lyapunov exponent at every orbit.
These are referred as singular-hyperbolic attractors or repellers.
Singular-hyperbolicity is a $C^1$-open condition.  
Every singular-hyperbolic attractor is Komuro-expansive and an homoclinic class (see \cite{vitorzeze}). 
So Theorem~\ref{thm:flows} implies there exists a $C^1$-open
and $C^\infty$-dense subset of $C^\infty$ singular-hyperbolic attractors 
whose centralizers are continuous reparameterizations of the flow.
\end{example}

One should mention that the singularities of $C^1$-robust Komuro expansive flows are hyperbolic
(cf. \cite{laura}). The following question arises naturally:

\medskip
\noindent {\bf Question~1:} Is the centralizer of Komuro expansive flows with isolated non-hyperbolic singularities trivial?
\medskip

The strategy used here can probably be applied to deal with other notions of expansiveness.
In \cite{artiguekinematicexp}, Artigue introduced some notions of expansiveness that we now recall.  
A flow is called \emph{kinematic expansive} if for all $\varepsilon > 0$ there exists $\delta > 0$ such that if 
$d(\varphi_t(x), \, \varphi_t(y)) < \delta$ for all $t \in \mathbb{R}$, then there exists $s \in (-\varepsilon, \, \varepsilon)$ with 
$y = \varphi_s(x)$. A flow is \emph{strong kinematic expansive} if every continuous reparameterization of the flow is 
kinematic expansive or, equivalently,  all topologically equivalent flows are kinematic expansive.
In \cite{artiguekinematicexp}, the author proves
\begin{equation}
\textrm{K-expansive} \Rightarrow \textrm{strong kinematic expansive} \Rightarrow \textrm{kinematic expansive}.
\end{equation}
Together with \eqref{compexpansiveness}, the later implies that $C$-expansiveness implies kinematic expansiveness. 
By \cite[Theorem 7.5]{artiguekinematicexp}, in the case of non-singular vector flows, the notions of 
$C^1$-robustly kinematic expansive, $C^1$-robustly strong kinematic expansive, $C^1$-robustly expansive, 
K-expansive or C-expansive flows coincide. Moreover, if this is the case such flows have a quasi-trivial centralizer \cite{oka}. 
The next example illustrates that kinematic expansive flows without singularities have quasi-trivial centralizer.

\begin{example}(Kinematic expansive flow with quasi-trivial centralizer) 
Consider $\mathbb S^1= \mathbb R / \mathbb Z$ and 
the flow $\varphi$ on $\mathbb{T}^2$ obtained as the suspension flow of the identity map on $\mathbb{S}^1$ by a smooth and positive smooth function $r: \mathbb S^1 \to (0,+\infty)$ without any plateau. The flow $\varphi$ is kinematic expansive but is not strong kinematic expansive. 
The proof of Lemma~\ref{singepsilon1} carries on for kinematic expansive flows, which guarantees that any element $\psi$
of the $C^1$-centralizer of $\varphi$ is a locally a reparameterization of $\varphi$. Moreover, since $\varphi$ has no singularities, the arguments of Subsection~\ref{sec:extend1} and ~\ref{sec:invarian1} yield
that $\psi$ is a linear reparameterization of $\varphi$. In other words, the centralizer of $\varphi$ is quasi-trivial.
\end{example}

Clearly 
the previous flow can be $C^1$-approximated by a flow that is not kinematic expansive.
In the following example we describe the centralizer of an example of strong kinematic expansive flow with a singularity.

\begin{example}\label{ex:Ar}
Consider an irrational flow on the two-dimensional torus $\mathbb{T}^2 = \mathbb{R}^2/\mathbb{Z}^2$ with vector field 
$X$ and let $f$ be any non-negative smooth function $f$ with just one zero at some point $p \in \mathbb{T}^2$. The flow $\varphi$ generated by the vector field $f X$ is strong kinematic expansive (cf. \cite[Example~2.8]{artiguekinematicexp}) (see
Figure~1 below). Since this flow has a non-hyperbolic singularity then Theorem~\ref{thm:flows} does not apply. In fact, although we do not need this here, it is not hard to show that $\varphi$ is not even Komuro expansive. 
We claim that $\mathcal{Z}^1(\varphi)$ is trivial. In fact, if $\sigma$ the unique singularity of $\varphi$ and 
$\psi \in \mathcal{Z}^1(\varphi)$ then $\psi_s(\sigma) = \sigma$ for all $s \in \mathbb{R}$. In other words, $\sigma$
is a singularity for $\psi$. Moreover, $\psi$ preserves the ($\varphi$-invariant) stable set   
$\mathcal{B}^s(\sigma) := \{y \in \mathbb{T}^2: \, d(\varphi_t(y), \sigma) \to 0$ as $t \rightarrow +\infty\}$. 
This set $\mathcal{B}^s(\sigma)$ is one dimensional it is formed by the orbit of any point in 
$\mathcal{B}^s(\sigma) \setminus\{\sigma\}$. 
As mentioned in the previous example, the arguments used to deduce the existence and uniqueness of 
a continuous and $\varphi$-invariant function $A: \mathbb T^2 \setminus \{\sigma\} \to \mathbb R$ so that
$\psi_t(x)=\varphi_{A(x) t} (x)$ for every $x\in \mathbb T^2 \setminus \{\sigma\}$ and $t\in \mathbb R$.
 Now, since $\mathcal{B}^s(\sigma)$ is dense in $\mathbb{T}^2$ and the function $A$ is constant along orbits of $\varphi$
 then it is constant in $\mathbb T^2 \setminus \{\sigma\}$. Thus, $A$ clearly extends to a constant function on the torus
 $\mathbb T^2$, which proves that there exists $c\in \mathbb R$ so that $\psi_t(x)=\varphi_{ct}(x)$ for every 
 $x\in \mathbb T^2$ and $t\in \mathbb R$. In other words, the $C^1$-centralizer of $\varphi$ is trivial.
\end{example}

In view of the previous example it seems natural to ask the following:

\medskip
\noindent {\bf Question 2:} Do all strong kinematic expansive with singularities have trivial centralizer?
\medskip

Finally, we describe the centralizer of Anosov $\mathbb R^d$-actions.

\begin{example}\label{corAn} (Anosov actions have quasi-trivial centralizer)
Let $M$ be a compact Riemannian manifold and let $\Phi: \mathbb{R}^d \times M \rightarrow M$ be 
an homogeneous Anosov action. Here we show that $\Phi$ has a quasi-trivial centralizer, thus
extending \cite{Kato}.
First we claim that every Anosov $\mathbb R^d$-action on a compact Riemannian manifold $M$ is kinematic expansive. 
This is probably well known but we could not find in the literature.
Let $\cF$ be the $\Phi$-orbit foliation. Then, there exists $v \in \mathbb{R}^d$ such that the diffeomorphism $\Phi_v$ is an Anosov element, hence normally hyperbolically. Let $\overline\delta>0$ be given by the plaque expansiveness of $(\Phi_v,\cF)$ (recall Subsection~\ref{unifh}). Given $\varepsilon > 0$ let $\delta = \min\{\varepsilon, \overline{\delta}\}>0$
and assume that $x, \, y \in M$ satisfy $d(\Phi_u(x), \, \Phi_{u}(y)) < \delta$ for every $u \in \mathbb{R}^d$. 
In particular, the orbits of $x,y$ by $\Phi_v$ differ by at most $\overline \delta$  
(since $d(\Phi_{n v}(x), \, \Phi_{n v}(y)) < \delta \leq \overline{\delta}$ for all $n \in \mathbb{Z}$). Moreover, 
the plaque expansiveness condition implies that $y \in \mathcal{F}(x)$. 
This proves that $y$ belongs to the orbit  of $x$ by $\Phi$ and that $d(x,y)<\vep$.
Thus there exists a vector $w \in \mathbb{R}^d$ such that $||w|| < \varepsilon$ and $y = \Phi_w(x)$, 
consequently the action $\Phi : \mathbb{R}^d \times M \rightarrow M$ is expansive.
Thus, the ingredients in the proof of Theorem~\ref{centacaoespansiva} allow to conclude that $\Phi$ has a quasi-trivial centralizer.
\end{example}

\vspace{.4cm}
\subsection*{Acknowledgements}
The authors are deeply grateful to the anonymous referee for very careful reading of the manuscript and useful advices that 
helped to improve the manuscript. This work is part of the first author's PhD thesis at UFBA. 
W.B. and P.V. were partially supported by BREUDS. J.R. and P.V. were partially supported by CMUP (UID/MAT/00144/2013), which is funded by FCT (Portugal) with national (MEC) and European structural funds (FEDER), under the partnership agreement PT2020. 
W.B and P.V. and grateful to Universidade do Porto, where this work was developed, for the warm hospitality and excellent research conditions.

\end{document}